\documentclass[letterpaper,11pt,reqno]{amsart} 
\usepackage[portrait,margin=2.5cm]{geometry}

\usepackage{systeme}
\usepackage{mathrsfs,xfrac} 
\usepackage[colorlinks=true,linkcolor=blue,citecolor=blue,urlcolor=blue]{hyperref} 
\usepackage{amsmath,amssymb,amsthm,amsfonts,amsbsy,latexsym,dsfont,color} 
\usepackage[numeric,initials,nobysame]{amsrefs} 
\usepackage[foot]{amsaddr}

\usepackage{booktabs} % For prettier tables
\usepackage{array}    % For more options on column specifications

\usepackage[utf8]{inputenc}
\usepackage[english]{babel}
\usepackage{comment}

\newcommand{\erdos}{Erd\H{o}s-R\'enyi}

\usepackage{enumerate}
\newenvironment{enumeratei}{\begin{enumerate}[\upshape (i)]}{\end{enumerate}}
\newenvironment{enumeratea}{\begin{enumerate}[\upshape (a)]}{\end{enumerate}}
\newenvironment{enumeraten}{\begin{enumerate}[\upshape 1.]}{\end{enumerate}}

\usepackage{paralist}

\newtheorem{definition}{Definition}
\newtheorem{lemma}{Lemma}
\newtheorem{remark}{Remark}

\newtheorem{thm}{Theorem}[section]
\newtheorem{lem}[thm]{Lemma}
\newtheorem{cor}[thm]{Corollary}
\newtheorem{prop}[thm]{Proposition}
\newtheorem{defn}[thm]{Definition}

\newtheorem{ass}[thm]{Assumption}

\newtheorem{conj}[thm]{Conjecture}   
\newtheorem{construction}[thm]{Construction}

\theoremstyle{definition}

\newtheorem{rem}[thm]{Remark}

\def\beq{ \begin{equation} }
\def\eeq{ \end{equation} }

\def\square{\vcenter{\vbox{\hrule height .4pt
  \hbox{\vrule width .4pt height 5pt \kern 5pt
        \vrule width .4pt} \hrule height .4pt}}}

\def\ep{\epsilon}

\def\ZZ{\mathbb{Z}}

\def\var{\hbox{var}\,}

\definecolor{darkblue}{rgb}{0,0,0.6}

\def\SWG{{\sf SWG}}

\def\ep{\varepsilon}
\def\1{\ind}
\def\CC{\mathcal{C}}
\def\deg{\mathrm{deg}}

\renewcommand{\le}{\leqslant} 
\renewcommand{\ge}{\geqslant} 
\renewcommand{\leq}{\leqslant} 
\renewcommand{\geq}{\geqslant}

\newcommand{\ind}{\mathds{1}}
\newcommand{\eps}{\varepsilon}

\newcommand{\set}[1]{\left\{#1\right\}}

\newcommand{\probc}{\stackrel{\mathrm{P}}{\longrightarrow}}

\def\qed{ \hfill $\blacksquare$}  
\newcommand{\cA}{\mathcal{A}}\newcommand{\cC}{\mathcal{C}}

\newcommand{\cG}{\mathcal{G}}

\newcommand{\cP}{\mathcal{P}}\newcommand{\cR}{\mathcal{R}}

\newcommand{\vB}{\mathbf{B}}
\newcommand{\vD}{\mathbf{D}}

\newcommand{\vd}{\mathbf{d}}

\newcommand{\vp}{\mathbf{p}}\newcommand{\vq}{\mathbf{q}}

\newcommand{\mvzero}{\boldsymbol{0}}

\newcommand{\mvpi}{\boldsymbol{\pi}}

\newcommand{\mvxi}{\boldsymbol{\xi}}

\newcommand{\bR}{\mathbb{R}}

\newcommand{\bZ}{\mathbb{Z}}        
\newcommand{\sP}{\mathscr{P}}

\newcommand{\ldown}{l^2_{\downarrow}}
\newcommand{\bfC}{\boldsymbol{C}}

\DeclareMathOperator{\E}{\mathds{E}}
\DeclareMathOperator{\pr}{\mathds{P}}

\DeclareMathOperator{\opt}{opt}  
\newcommand{\BP}{{\sf BP}}
\newcommand{\sss}{\scriptscriptstyle}
\newcommand{\convd}{\stackrel{d}{\longrightarrow}}
\newcommand{\convp}{\stackrel{P}{\longrightarrow}}

\DeclareMathOperator{\CM}{CM}

\usepackage[textsize=small]{todonotes}       % includes TO DO LIST

\definecolor{aqua}{rgb}{0.0, 1.0, 1.0}
\definecolor{boo}{rgb}{1.0, 0.0, 1.0}
\definecolor{stred}{rgb}{1.0, 0.44, 0.37}

\newcommand{\stod}{\preceq_{\mathrm{st}}}
\newcommand{\sustod}{\succeq_{\mathrm{st}}}

\begin{document}

\title[Critical FPP on random graphs]{Critical first passage percolation on random graphs}

\author[Bhamidi]{Shankar Bhamidi$^*$}
\email{bhamidi@email.unc.edu}
\address{$^*$ Department of Statistics and Operations Research, University of North Carolina at Chapel Hill, USA}

\author[Durrett]{Rick Durrett$^\dagger$}
\email{rtd@math.duke.edu}
\address{$\dagger$ James B. Duke Emeritus Professor of Mathematics, Duke University, USA}

\author[Huang]{Xiangying Huang$^*$}
\email{zoehuang@unc.edu}
%\address{$^1$Department of Statistics and Operations Research, University of North Carolina at Chapel Hill, US}

%\date{\today}						

\begin{abstract}
In 1999, Zhang \cite{Zhang99} proved that, for first passage percolation on the square lattice $\bZ^2$ with \emph{i.i.d.}~non-negative edge weights, if the probability that the passage time distribution of an edge $\pr(t_e = 0) =1/2 $, the critical value for bond percolation on $\bZ^2$,  then the passage time from the origin $\mvzero$ to the boundary of $[-n,n]^2$ may converge to $\infty$ or stay bounded depending on the nature of the distribution of $t_e$ close to zero. In 2017, Damron, Lam, and Wang \cite{Damron} gave an easily checkable necessary and sufficient condition for the passage time to remain bounded. Concurrently, there has been tremendous growth in the study of weak and strong disorder on random graph models. Standard first passage percolation with {\it strictly positive} edge weights provides insight in the weak disorder regime e.g. \cite{BHH10b}. Critical percolation on such graphs provides information on the strong disorder (namely the minimal spanning tree) regime \cite{addario2017scaling}. 
   Here we consider the analogous problem of Zhang but now for a sequence of random graphs $\set{\cG_n:n\geq 1}$ generated by a supercritical configuration model with a fixed degree distribution. Let $p_c$ denote the associated critical percolation parameter, and suppose each edge  $e\in \mathcal{E}(\cG_n)$ has weight $t_e \sim p_c \delta_0 +(1-p_c)\delta_{F_\zeta}$ where $F_\zeta$ is the cdf of a random variable $\zeta$ supported on $(0,\infty)$. The main question of interest is: when does the passage time between two randomly chosen vertices have a limit in distribution in the large network $n\to \infty$ limit? There are interesting similarities between the answers on $\ZZ^2$ and on random graphs, but it is easier for the passage times on random graphs to stay bounded. 

 \end{abstract}

\maketitle

\section{Introduction}

First passage percolation (FPP), namely networks where each edge $e$ has an associated non-negative passage time $t_e$, sampled in an \emph{i.i.d.}~fashion from a distribution $F$,  has been extensively studied on $\ZZ^d$ and on random graphs. For a comprehensive overview of established results, see \cite{50yrs, vdHSF}. The main goal of this paper is to investigate critical FPP, where the passage time is equal to zero with probability $p_c$, the critical value for bond percolation on the graph.

We begin by describing the work of Zhang \cite{Zhang99} and Damron, Lam, and Wang \cite{Damron} on $\ZZ^2$. In this case, $p_c=1/2$ and it is known that there is no percolation (existence of an infinite connected component) at the critical value \cite{kesten1980critical}. 
Let $B_n= [-n,n]^2\cap \bZ^2$ and $\partial B_n=\{x\in \bZ^2: \|x\|_\infty=n\}$ denote its boundary. The passage time from 0 to  $\partial B_n$, denoted by $T(0,\partial B_n)$, refers to the minimum passage time associated with $\{t_e: e\in \mathcal{E}(\bZ^2)\}$ over all paths between 0 and $\partial B_n$. 
For a given passage time distribution $F$, as $n \to\infty$ the passage time $T(0,\partial B_n)$ increases to a limit 
$$\rho(F):=\lim_{n\to\infty} T(0,\partial B_n),$$
which is a random variable and Kolmogorov's zero-one law tells us that $\pr(\rho(F)=\infty)\in \{0,1\}$.

One major question is to understand  for which distributions $F$ we have $\rho(F) < \infty$, meaning that $T(0,\partial B_n)$ stays bounded as $n \to\infty$.  It is not hard to show that if $F$ has no mass in $(0,\delta]$ for some $\delta>0$ then $T(0,\partial B_n) \to \infty$ as $n\to\infty$, and the main determining factor is the amount of mass near zero. Zhang created two families of distributions to study the question. Here for parameters $a,b>0$: 
\begin{equation}
F_a(x) = \begin{cases} 0 & x <0 \\ p_c + x^a & 0 \le x^a \le 1-p_c \\
1 & p_c + x^a \ge 1, \end{cases} 
\label{eqn:Fa-def}
\end{equation}

\begin{equation}
G_b(x) = \begin{cases} 0 & x <0 \\ p_c + \exp(-1/x^b) & 0 \le \exp(-1/x^b) \le 1-p_c \\
1 & p_c + \exp(-1/x^b) \ge 1. \end{cases}
\label{eqn:Gb-def}
\end{equation}

\noindent
In \cite{Zhang99} Zhang  proved that:

\begin{enumeratea}
\item
If $a$ is sufficiently small then $\rho(F_a) < \infty$ almost surely;

\item
If $b>1$ then $\rho(G_b) = \infty$ almost surely.
\end{enumeratea}

\noindent
He conjectured that $\sup\{ a>0: \rho(F_a) < \infty \}< \infty$. Damron, Lam, and Wang in \cite{Damron} proved a definitive result which showed that the conjecture was not correct. Defining $\eta_0 = \sup\{ \eta\ge 0: E(t_e^{\eta/4}) < \infty\}$ and assuming $\eta_0>1$, \cite{Damron} proved that there exist constants $0 < C_1(F) \leq C_2(F) <\infty$ such that, 
\begin{equation}
 \label{damron-et-al}
 C_1 \sum_{k=2}^n F^{-1}(p_c + 2^{-k}) \le \E(T(0,\partial B_{2^n})) \le C_2 \sum_{k=2}^n F^{-1}(p_c + 2^{-k}). 
 \end{equation}

\noindent
This implies that $\rho(F) < \infty$ if and only if 
$ \sum_{k=2}^\infty F^{-1}(p_c + 2^{-k})  < \infty$. 
It follows easily from their criterion that:

\begin{enumeratea}
\item
$\rho(F_a) < \infty$ a.s. for all $0< a < \infty$

\item
$\rho(G_b) = \infty$ a.s. if and only if $b \ge 1$.

\end{enumeratea}

\noindent
Before moving on to our setting we should observe that the proofs of the above results rely on special properties of two dimensional percolation, and the questions in $d>2$ are difficult open problems.

\subsection{Random graph setting}
\label{sec:rg-cm}

We will use the standard notation $ \convp, \overset{d}{\longrightarrow}$ to denote convergence in probability and in distribution, respectively. A sequence of events $(E_n)_{n\geq 1}$ is said to occur with high probability (whp) with respect to probability measures $(\pr_n)_{n\geq 1}$ if $\pr_n(E_n)\to 1$. We write $f_n = O_{\pr}(g_n)$ if $(|f_n|/|g_n|)_{n\geq 1}$ is tight; $f_n=\Omega_{\pr}(g_n)$ if $(|g_n|/|f_n|)_{n\geq 1}$ is tight; $f_n=\Theta_{\pr}(g_n)$ if $f_n=O_{\pr}(g_n)$ and $f_n=\Omega_{\pr}(g_n)$; $f_n =o_{\pr}(g_n)$ if $(|f_n |/|g_n |)_{n\geq 1}$ converges in probability to zero.

We now describe the network model of interest in this paper.

\begin{figure}[htbp]

\def\degrees{{3,2,2,2,1}}
  % blue/.style={circle, inner sep=3pt, ball color=blue!40},

% Initial configuration with half-edges
\begin{tikzpicture}[scale=.9, transform shape, every node/.style={circle, ball color=blue!30, minimum size=6mm, inner sep=3pt}]
    % Nodes
    \node (A) at (0,0) {1};
    \node (B) at (2,1.5) {2};
    \node (C) at (2,-1.5) {3};
    \node (D) at (4,0) {4};
    \node (E) at (6,0) {5};

    % Half-edges
    % Node A (3 half-edges)
    \foreach \i in {1,2,3}
        \draw[thick, red] (A) -- ++(-1.5, {(\i-2)*0.5});
    % Node B (2 half-edges)
    \foreach \i in {1,2}
        \draw[thick, red] (B) -- ++(1, {(\i-1)*0.5});
    % Node C (2 half-edges)
    \foreach \i in {1,2}
        \draw[thick, red] (C) -- ++(1, {(\i-1)*0.5});
    % Node D (2 half-edges)
    \foreach \i in {1,2,3,4}
        \draw[thick, red] (D) -- ++(1, {(\i-2.5)*0.5});
    % Node E (1 half-edge)
    \draw[thick, red] (E) -- ++(1, 0);

    % Labels
    % \node[draw=none, fill=none, scale=0.8] at (-2, 1) {Initial Configuration};
\end{tikzpicture}
 \hspace{1in}
% Final graph
\begin{tikzpicture}[scale=.74, transform shape, every node/.style={circle, ball color=blue!30, minimum size=6mm, inner sep=3pt}, every edge/.style={draw, thick}]
    % Nodes
    \node (A) at (0,0) {1};
    \node (B) at (2,1.5) {2};
    \node (C) at (2,-1.5) {3};
    \node (D) at (4,0) {4};
    \node (E) at (6,0) {5};

    % Edges
    \draw[black, very thick] (A) -- (B);
    \draw[black, very thick] (A) -- (C);
    \draw[black, very thick] (A) -- (D);
    \draw[black, very thick] (B) -- (D);
    \draw[black, very thick] (C) -- (D);
    \draw[black, very thick] (D) -- (E);

    % Labels
    % \node[draw=none, fill=none, scale=0.8] at (-2, 1) {Final Graph};
\end{tikzpicture}
\label{fig:CM}
\caption{Example of the Configuration model, with the picture on the left showing the initial set of half edges or stubs, while the right picture showing an example of the realization of the matching of these half edges. }
\end{figure}
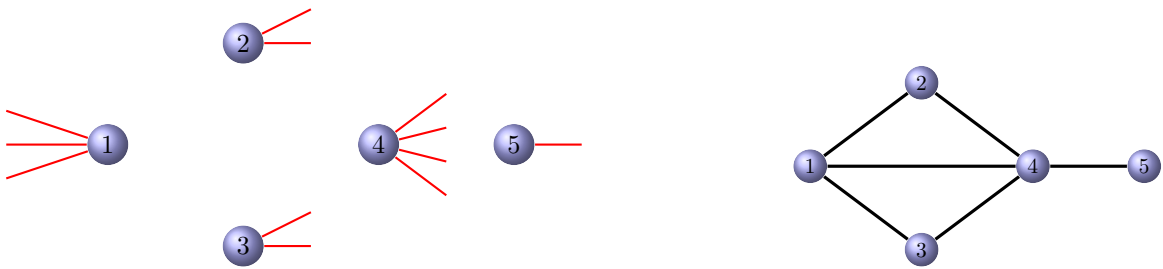

 Write $[n]:=\set{1,2,\ldots n}$. Fix a probability mass function (pmf) $\vp = \set{p_k:k\geq 0}$ on $\bZ_+$. Let $\set{\cG_n:n\geq 1}$ be a sequence of graphs generated using the configuration model with degree distribution $\vp$, namely $\cG_n$ is a random graph with vertex set $[n]$ constructed via:
\begin{enumeratea}
\item Let $\set{d_u:u\in [n]} \sim_{i.i.d.} \vp$. Think of vertex $u$ starting with $d_u$ half-edges and assume $\sum_{u\in [n]} d_u$ is even, else modify the degree of vertex $v=n$ from $d_n$ to $d_n+1$. 
\item Starting with the $\sum_{u\in [n]} d_u$ active half-edges (sometimes called \emph{stubs}), sequentially and at each stage randomly select two half-edges amongst the current set of active half-edges, merge them to form a full edge, and remove these half-edges from the active set. Continue this process until all active half-edges have been exhausted. 
\end{enumeratea}
See \cites{bollobas1980probabilistic,bender1978asymptotic,molloy1995critical} and the subsequent citations for the relevance of this model in probabilistic combinatorics and its applications. Figure \ref{fig:CM} gives a visualization of this construction.

 We write $\cG_n \sim \CM_n(\vp)$ for the corresponding graph. Now assume every edge $e\in \mathcal{E}(\cG_n)$ has a random edge weight $t_e$ with distribution $F$, independent across edges.   Write $D\sim \vp$ and $t_e\sim F$ for generic random variables representing the degree and edge-weight distribution respectively.   
 \begin{ass}[Assumptions on degree and edge weights]
 \label{ass}

 \ 
 \begin{enumerate}[({A}1)]
 \item {\bf Supercriticality:} Assume $\pr(D\geq 2) =1, \pr(D=2) <1$. This automatically implies that the fraction of the vertices in the giant component goes to $1$ as $n \to\infty$ \cites{molloy1995critical,molloy1998size}. 
 \item {\bf Moment conditions:} Assume $\E[D^{3+\eta}]<\infty$ for some $\eta>0$. Write $\nu = \sum_{k=2}^\infty k(k-1)p_k/\sum_k k p_k$. By our assumptions, $1<  \nu < \infty$. 

 \item {\bf Critical percolation and passage time mass at zero:} Assume that the edge-weight distribution satisfies $\pr(t_e = 0) = 1/\nu$. 
 \end{enumerate}
 \end{ass}
 
We now explain why (A3) is called the ``critical percolation'' regime. First, the constant $\nu$ in (A2) is the mean of the \textit{size-biased} distribution of $\vp$, denoted by $\vq=\{q_k:k\geq 0\}$ with
\begin{equation} \label{size_biased} 
q_k = \frac{(k+1)p_{k+1}}{\sum_j j p_j},\qquad k\geq0.
\end{equation}

 Next, it is known \cites{molloy1995critical,janson2009new} that the critical percolation threshold $p_c$ on $\cG_n\sim \CM_n(\vp)$ is $p_c  = 1/\nu$, namely for independent edge percolation on $\cG_n$ with parameter $p$:
 \begin{enumeratea}
  \item When $p< p_c$ the size of the largest component $\cC_{\sss(1)} = o_{\pr}(n)$. 
  \item For $p> p_c$, $\cC_{\sss(1)}/n \convp f(p)$ for a strictly positive function $f(\cdot)$. 
  \end{enumeratea}  
There is in fact an entire critical scaling window of the form $p_c(\lambda) = 1/\nu + \lambda/n^{1/3}$ for any $\lambda \in \bR$ \cites{joseph,dhara2017critical}. We discuss the insensitivity of our results on this second order perturbation in Section \ref{sec:disc} namely the lack of dependence of the limits derived in this paper on the value $\lambda$. In Section \ref{sec:disc} we also describe how various other functionals of the model such as maximal sizes of so-called \emph{zero-weight clusters} sensitively do depend on $\lambda$.

\begin{defn}[Main functionals of interest]
\label{def:main-func}
For any pair of vertices $u\neq v \in [n]$, let $\sP_n(u,v)$ denote the collection of all (self-avoiding) paths from $i,j$. For any path $\mvpi \in \sP_n(u,v)$ let the total passage time using this path be given by $T(\mvpi) = \sum_{e\in \mvpi} t_e$ and let $|\mvpi|$ denote the number of edges in $\mvpi$.  Let $T_n(u,v) = \min_{\mvpi \in \sP_n(u,v)} T(\mvpi)$ denote the optimal passage time between these two vertices and 
\[\sP_{\opt}(u,v) = \set{\mvpi \in \sP_n(u,v): \mvpi = \arg \min T(\mvpi) },\]
denote the collection of optimal paths. Let 
$H_n(u,v) = \min \set{|\mvpi|: \mvpi \in \sP_{\opt}(u,v) },$
denote the hopcount, namely the minimal number of edges amongst all optimal paths between $u,v$. 
\end{defn}
 By symmetry in the construction, for any $u,v\in [n]$ such that $u\neq v$, $T_n(u,v) \stackrel{d}{=} T_n(1,2)$ and similarly for the hopcount. 

\subsection{Setting the stage}
\label{sec:stage}
Before describing our results, we briefly discuss what is known in related settings for such edge disorder models on random graphs:
\begin{enumeratei}
\item Suppose each edge weight is strictly positive (i.e., no mass at zero), then under some further technical conditions, \cite{BHH17} shows that there are (model dependent) strictly positive constants $\alpha, \beta, \gamma >0$ and non-trivial, model dependent, finite random variables $\Xi$ such that 
\begin{equation} 
\label{eqn:1050}
T_n(1,2) - \frac{1}{\alpha}\log{n} \convd \Xi, \qquad \frac{H_n(1,2) - \beta \log{n}}{\sqrt{\gamma \log{n}}} \convd Z,
\end{equation}
where $Z\sim N(0,1)$. The same arguments as used in this paper coupled with \cite{BHH17} should imply that (at least up to first order), the above asymptotics should continue to hold as long as $\pr(t_e =0) < p_c$. On the other extreme, if $\pr(t_e = 0) > p_c$, then a zero-weight cluster percolates, i.e., there is a giant component of all zero-weight edges so $\liminf_{n\to\infty} \pr(T_n(1,2) = 0) > 0$. 
\item Now suppose $\pr(t_e = 0) = p_c$. Suppose we first lay down the zero-weight  edges and initially let all other edges have weight $\infty$. Then the network gets disconnected into a collection of zero-weight ``super conducting'' components (in the sense that it takes zero time to traverse them). Under our moment conditions the number of vertices in the {\bf maximal} zero-weight component scale like $n^{2/3}$ while their geometry, in terms of graph distance between typical points scales like $n^{1/3}$ \cites{dhara2017critical,goldschmidt2020stable}; one impetus for deriving such detailed results is understanding the nature of the minimal spanning tree on such graphs with positive random edge lengths with a continuous distribution \cite{addario2017scaling}.
\end{enumeratei}

\section{Results}
\label{sec:results}
For convenience we write the edge distribution as
\beq
\label{edg:decomp}
F(t) = \begin{cases}  p_c &  \text{ for }t=0, \\ p_c + (1-p_c) F_{\zeta}(t) & \text{ for } t >0 \end{cases}
\eeq
where $F_\zeta(t)$ represents the strictly \textit{positive part} of the edge passage time and satisfies $F_\zeta(0) = 0$. 

\subsection{Necessary and sufficient conditions for bounded passage times}
\label{sec:nsc-cond}

The positive part $F_\zeta$  of the edge-weight distribution as in  \eqref{edg:decomp} is said to be {\bf min-summable} if there exists some $\ep>0$ such that, 
$$
\int_0^\ep F_\zeta^{-1}( \exp(-1/y)) \frac{dy}{y} < \infty,
$$
or equivalently
\beq
\int_{1/\ep}^\infty F_\zeta^{-1}(e^{-u}) \frac{du}{u} < \infty,
\label{PLexp}
\eeq
see \cite[Lemma 5.8]{Kom}. This condition is adapted from the min-summability criterion for branching processes, see Theorem 5.7  in \cite{Kom}. The phrase ``min -summable" comes from the equivalent condition in \cite[Eqn 5.9]{Kom} that shows that \eqref{PLexp} is equivalent to the condition that for any $\alpha \in (0,1)$ and 
\[\sum_{n=1}^\infty F_\zeta^{-1}\left(\frac{1}{e^{1/\alpha^n}}\right) < \infty. \]
This condition should be contrasted with the condition \eqref{damron-et-al} from \cite{Damron} on $\bZ^2$.

\begin{thm}
\label{thm:main}
Under Assumptions \ref{ass}, typical passage times on $\cG_n \sim \CM_n(\vp)$ are bounded, i.e., $T_n(1,2) = O_{\pr}(1)$, if and only if the positive part of the passage time $F_\zeta$ is {\bf min-summable}. Further, if $F_\zeta$ is min-summable then 
$$T_n(1,2) \convd \Xi,$$ 
where $\Xi = \Xi(\vp, \zeta)$ is a model dependent finite random variable. 
\end{thm}

\begin{rem}
Our proofs will show that $\Xi = V_1+ V_2$ where $V_i$ are independent copies of the explosion time for associated age-dependent branching process defined in Section \ref{sec:adbp}. Starting with \cite{janson1999one} studying first passage percolation for the complete graph, and continuing in a number of sparse random graph models see e.g.  \cites{BHH10b,BHH11,BHH17,Baroni,baroni2019tight,van2017explosion,van2015fixed,vdHSF} and the references therein,  this is part of a body of work showing similar distributional limits for the optimal passage time for first passage percolation on a host of random graph models.  
\end{rem}

\begin{cor}
\label{cor:examples}
Under the assumptions of Theorem \ref{thm:main}, consider the following special cases for the positive part of the edge distribution:
\begin{enumeratea}
\item \label{first} Suppose that $(t_e|t_e>0)$ has the exponential distribution with mean 1.
Then $T_n(1,2) \convd \Xi(\vp, {\mathrm{Exp}(1)})$ as $n\to\infty$. 
\item \label{gamma}
Suppose that $t_e$ has distribution function $F_{a}(\cdot)$ as in \eqref{eqn:Fa-def} for $a>0$. Then $T_n(1,2)\convd \Xi(\vp, {F_{a}})$  as $n\to\infty$. 
\item \label{Gbth} Suppose that $t_e$ has distribution function $G_{b}(\cdot)$ as in \eqref{eqn:Gb-def}   for $b> 0$. Then for any $b\in(0,\infty)$ the distribution of $T_n(1,2) \convd \Xi(\vp, {G_{b}})$ as $n\to\infty$. \end{enumeratea}

\end{cor}

\begin{rem}
Part (a) contrasts with the case when $t_e= \mathrm{Exp}(1)$ on the configuration model  $\cG_n \sim \CM_n(\vp)$, where by \cite{BHH10b}, $T_n(1,2)  = \Theta_{\pr}(\log{n})$. We also compare parts (b) and (c) with the results for critical FPP on $\bZ^2$ from \cite{Damron}. In part (b), the conclusion is similar to \cite{Damron} which shows that $\rho(F_a) < \infty$ for all $0 < a < \infty$. However, part (c) contrasts with  \cite{Damron}, where $b < 1$ is required for the passage times to remain bounded, i.e., $\rho(G_b) <\infty$ iff $b<1$.
\end{rem}

Theorem \ref{thm:main} gives necessary and sufficient conditions for boundedness of the passage times. Thus a natural question is to find an explicit examples where typical passage times are {\bf not} bounded. The next result, which also follows from Theorem \ref{thm:main}, demonstrates such scenarios.

\begin{cor} \label{dexp}
Under the assumptions of Theorem \ref{thm:main}, consider the following special cases for the positive part of the edge distribution:
\begin{enumeratea}
\item Suppose  $(t_e|t_e>0)$ has cumulative distribution function $H_\gamma(t)=e \cdot \exp(-\exp(1/t^\gamma))$ for $0< t \leq 1$. Then as $n\to\infty$, 
\begin{enumeratei}
\item  For $\gamma < 1$, $ T_n(1,2) \convd \Xi $ where $\Xi = \Xi(\vp, H_\gamma) < \infty$ a.s., 
\item  For $\gamma \geq 1$, $T_n(1,2)\convd \infty$.
\end{enumeratei}
\item Suppose $t_e\overset{d}{=}p_c \delta_0+(1-p_c)\delta_1$, namely the passage time distribution has mass at either zero or one. Then $T_n(1,2)\convd \infty$.

\end{enumeratea}
\end{cor} 
\begin{conj}\label{lb_passage_time}
In the setting of Corollary \ref{dexp} (a) part (ii), we conjecture that there exists a constant $c>0$ such that with high probability $T_n(1,2) \geq c (\log\log n)^{1-(1/\gamma)}$. 
\end{conj}
We discuss our reasons for believing this conjecture in Section \ref{sec:pf-corr}.

\begin{conj}
\label{conj:zero-one}
	Example (b) in the above Corollary \ref{dexp} is closely connected to a question studied in \cite[Theorem 1.8]{Baroni}. In this setting we conjecture, 
$$
\frac{T_n(1,2)}{\log\log n} \probc \frac{2}{\log(2)}.
$$
\end{conj}
\noindent We will discuss this Conjecture in Section \ref{sec:disc}, below \eqref{baroni}.
 
Table \ref{tab:compare} below summarizes settings in which the passage time remains bounded on $\ZZ^2$ compared to the random graphs in this paper, satisfying (A1)--(A3), for different edge-weight distributions in the critical regime of zero-weight edges.

\begin{table}[h]
    \centering
    \begin{tabular}{>{\raggedright\arraybackslash}p{4cm} >{\centering\arraybackslash}p{3cm} >{\centering\arraybackslash}p{4cm}}
        \toprule
        \textbf{$(t_e|t_e>0)$} & \textbf{$\ZZ^2$} & \textbf{Config. Model} \\
        \midrule
        $(t \wedge 1)^a$ & always & always \\
        $\exp(-1/t^b)$ &  $b<1$ & always \\
        $e\cdot \exp(-\exp(1/t^\gamma))$  & never & $\gamma < 1$ \\
        \bottomrule
    \end{tabular}
    \caption{Conditions for bounded passage times on $\bZ^2$ versus $\CM$}
    \label{tab:compare}
\end{table}

\subsection{Flooding time behavior}
Another interesting question, in scenarios where passage times are generally bounded, involves investigating the flooding time --- the time it takes for information originating from a fixed vertex, say $u=1$, to reach all other vertices. 
Unlike typical passage times, the flooding time is susceptible to the influence of ``extremal'' structures in the graph $\mathcal{G}_n$, such as the long ``strands" (or isolated paths) of length $\Theta(\log n)$. 

We will use the phrase \emph{isolated path} to refer to a path consisting of vertices of degree 2. The length of an isolated path is the number of vertices contained in this path. We first state a known result on existence of long isolated paths, paraphrasing \cite[Theorem 5.14]{fernholz2007diameter} to the setting in this paper. 
\begin{prop}\label{isolated_path}
Under Assumptions \ref{ass}, for any $\ep>0$, with high probability there exists an isolated path in $\mathcal{G}_n\sim \CM_n(\vp)$ of length at least $\left(\frac{1}{-2\log q_1}-\ep\right)\log n$, where $q_1$ is as defined in \eqref{size_biased}.
\end{prop}

Proposition \ref{isolated_path} leads to our final result, which shows that uniform boundedness of $T_n(1,v)$ as $v$ ranges over $[n]$ is not possible even if typical distances are bounded.

\begin{thm}
\label{thm:extremal}
Suppose Assumptions \ref{ass} are satisfied and $\E[t_e] \in (0,\infty)$. There exists $c>0$ so that, 
$$
\pr( \max_{v\in [n]} T_n(1,v) \ge c\log n ) \to 1.
$$
\end{thm}

\section{Discussion and related work}
\label{sec:disc}

\subsection{Proof techniques and overview: explosiveness of branching processes}
We start by giving some intuition to explain our results, setting the stage for a clearer path to the rest of this section, as well as the proofs. For now, we postpone rigorous statements, leaving precise details to Section \ref{sec:preliminary} and subsequent sections.

To understand intuitively the effect of having zero-weight edges in the FPP on the configuration model $\CM_n(\vp)$, we perform a two-stage construction:
\begin{enumeratea}
\item Given a graph $\mathcal{G}_n$ constructed from $\CM_n(\vp)$, each edge is independently retained with probability $p_c=1/\nu$ (otherwise it is deleted).
The resulting graph will be denoted by $\CM_n(\vp,p_c)$.  In the large network limit, the {\bf number of vertices} in a typical zero-weight cluster  follows a power-law distribution with tail exponent $\tau = 5/2$ (Lemma \ref{size_cluster}). The edges within such a cluster correspond solely to the zero-weight edges, where as the deleted edges (originally present in $\CM_n(\vp)$) correspond to positive-weight edges, where the edge weights are drawn independently from $\zeta \sim F_\zeta$ as in \eqref{edg:decomp}.
\item Collapse every zero-weight clusters into a supervertex whose ``degree" in the collapsed graph is given by the number of positive-weights edges emanating from it. The number of these positive-weight edges also follows a power-law distribution with the same exponent $5/2$ (Lemma \ref{cluster_degree}). 

\end{enumeratea}

Thus conceptually, \emph{critical} first passage percolation on the original graph should have similiar behavior to first passage percolation with strictly positive edge-weights with distribution $F_\zeta$, on a related configuration model with a heavy tailed degree distribution with degree exponent $\tau = 5/2 \in (2,3)$. In the approximation model, the vertex set is the set of zero-weight clusters, while the ``degree'' is the number of positive-weight half edges emanating from each zero-weight cluster. Thus, while the original network model might not have  heavy tailed degree distribution, the incorporation of these super conducting (namely zero passage time) edges leads to an intermediate state with heavy tails.

First passage percolation on heavy tailed random graph models,  with exponential edge passage times, were studied in \cite[Theorem 3.2]{BHH10b}. Here one set of results showed that  on power law graphs with $\tau\in(2,3)$ namely where the degree distribution satisfied, 
\beq
c_1 x^{-(\tau-1)} \le 1 - P( D \ge x) \le c_2 x^{-(\tau-1)}
 \quad\hbox{for $x\ge 1$}
\label{tau23}
\eeq
and $t_e\sim \mathrm{Exp}(1)$, the passage time between typical vertices $T_n(1,2) \convd \theta^{(1)}_\infty+\theta^{(2)}_\infty$ as $n\to\infty$, where $\theta^{(1)}_\infty,\theta^{(2)}_\infty$ are independent copies of the explosion time for the branching processes that approximate the local neighborhoods of the the two vertices. 
  Baroni et al. \cite{Baroni} extended this result to more general edge-weight distributions, proving $T_n(1,2) \convd \theta^{(1)}_\infty+\theta^{(2)}_\infty$ provided the  age-dependent branching processes approximating local neighborhoods explode in finite time.

In a different direction, in the same paper, Theorem 1.8 in \cite{Baroni} shows that for such heavy tailed random graphs with $\tau\in (2,3)$, if the edge passage times are bounded below by a positive constant (e.g. of the form $t_e = 1+X$ for some $X\geq 0 $ where the infimum of the support of $X = 0$) then 
\beq\label{baroni}
\frac{T_n(1,2)}{\log\log n} \probc \frac{2}{|\log(\tau-2)|}.
\eeq
This is the rationale for Conjecture \ref{conj:zero-one}.

The two papers on explosion for branching processes particularly influential for this paper is the foundational work by Grey in the 1970s \cite{Grey} and the comprehensive recent survey by Komj{\' a}thy \cite{Kom}. Applications of this methodology to first passage percolation in random graphs such as the configuration model was initiated in \cite{BHH10b}, but the role and impact of explosivity of local neighborhoods was explored by van der Hofstad, Komj{\' a}thy and co-authors in a sequence of papers e.g. \cites{van2015fixed,van2017explosion,baroni2019tight};  \cite{Baroni} was particularly influential in the technical portion of this paper.

\subsection{Critical scaling window}
As described above, one reason for the explosion of the associated branching processes, is the heavy tailed nature of the zero-weight clusters when the probability of a zero-weight edge equals $p_c = 1/\nu$. This is directly related to $p_c$ being the critical percolation threshold for the configuration model. It has been known, since the work of \erdos, that for many families of random graphs,   there is an entire critical scaling window. For the configuration model with finite third moments, the critical scaling window is of the form $p_c(\lambda) =1/ \nu+ \lambda n^{-1/3}$ for $\lambda \in \bR$. For percolation on the configuration model with this choice of edge retention probability,  order the sizes of the connected components as $\cC_{\sss(1),n}(\lambda) \geq \cC_{\sss(2), n}(\lambda) \geq \cdots$. View the entire vector $\bfC_{\sss(n)}(\lambda) = (\cC_{\sss(1),n}(\lambda), \cC_{\sss(2),n}(\lambda), \ldots,)$ as an element of $\ldown$:
\begin{equation}
    \ldown = \{(x_1,x_2,\ldots): x_1\geq x_2\geq \cdots \geq 0, \sum_i x_i^2< \infty\},
    \label{eqn:ldown}
\end{equation}
 via adjoining an infinite collection of zeros. Then it is known \cites{joseph,dhara2017critical} that there exists an infinite random vector $\mvxi(\lambda)$ with all positive entries  such that $n^{-2/3}\bfC(\lambda) \convd \mvxi(\lambda)$ on $\ldown$.  Furthermore, each of these maximal components, viewed as metric spaces with edges rescaled by $n^{-1/3}$, converge to limiting random compact metric spaces \cite{bhamidi2020geometry}. The bottom line is that, in the critical regime, major functionals of interest sensitively depend on the location (namely value of $\lambda$) in the critical regime. However, the critical FPP model considered in this paper, or rather the functionals we study, namely minimal passage times, are oblivious to these second order fluctuations as they depend on the local weak convergence of neigborhoods around typical vertices, which are not affected by the choice of $\lambda$. See Remark \ref{rem:crit} for a technical explanation for this robustness. We imagine that other functionals, such as extremal passage times (e.g., the flooding time in Theorem \ref{thm:extremal}) should feel the effect of this perturbation. In work in progress, we study one such functional, the hopcount, introduced in Definition \ref{def:main-func}.

% \clearp
\section{Proofs: Preliminary constructions and estimates}
\label{sec:preliminary}
\subsection{Overview}
\label{sec:proof-overview}
We start in Sections \ref{sec:adbp} and \ref{sec:comp-bp} by defining the main probabilistic tool of relevance for this paper, age dependent branching processes and their connection to first passage percolation. In Section \ref{sec:shortest-wg} we describe the sequential exploration of the shortest weight graph in the configuration model, starting from a typical vertex. Locally (i.e., at least until the exploration does not grow too large), in the $n\to\infty$ limit, this exploration should look like the branching process described in the prior two Sections.  To obtain the shortest path between two typical vertices, the now standard approach is to grow such shortest weight graphs simultaneously from these vertices till the two flow clusters ``merge'' (reach a vertex that is in both clusters) resulting in the optimal path. Section \ref{sec:proof-coupling} begins the technical heart of the paper where in part we derive error bounds tight enough to understand the strength of coupling between the shortest weight graphs from the two fixed vertices and corresponding age-dependent branching processes. Section \ref{sec:PfTh1} uses the estimates and constructions in this Section to complete the proof of Theorem \ref{thm:main}.  Section \ref{sec:pf-remaining} contains proofs of the remaining results.

\subsection{Explosive age-dependent Branching Process}
\label{sec:adbp}
We give a concise description of the key probabilistic tool employed in this paper, referring the interested reader to e.g. \cites{jagers,athreya2004branching,harris1963theory} for comprehensive treatments. 

\begin{definition}[Age-dependent branching process]\label{age_dependent_BP}
Fix distribution $F$ on $\bR_+$ and probability mass function on $\bZ_+$ with generating function $h$.  An age dependent branching process ${\sf BP}(F,h)$ is a branching process in continuous time,  starting with a single individual at time $t=0$ with individuals having random life-lengths with distribution function $F$ such that at death, an individual produces offspring of random size with probability generating function $h$. All life-lengths and family sizes are independent of each other. We assume that the process starts at time $t=0$ with one individual of age zero. 

% This process will be denoted by $\BP(F,h)$.
\end{definition}

When individuals are allowed to give births immediately, i.e., when $F(0)>0$, we can start the branching process with one individual at time $t=0$ and have multiple births in the population at time $t=0$. This creates ambiguity in describing the size and structure of the population at a given time $t$. We introduce an equivalent characterization (from \cite{Grey}) of the  branching process ${\sf BP}(F,h)$ as follows. Define 
\beq\label{birth_time}
F^*(t)=
\begin{cases}
\frac{F(t)-F(0)}{1-F(0)} &\text{ if }F(0)>0,\\
F(t) &\text{ if }F(0)=0
\end{cases}
\eeq 
so that $F^*(0)=0$.  Let $h^*(\cdot)$ denote the probability generating function, arising as the unique solution of, 
\beq\label{adapted_offspring}
h^*(s)=h\left( (1-F(0))s+F(0)h^*(s)\right),
\eeq
for each $s$. In the context of the FPP model in this paper, $F^* = F_{\zeta}$ namely the positive part of the edge weight.  As described in \cite{Grey}, the process ${\sf BP}(F,h)$ is equivalent to ${\sf BP}(F^*,h^*)$ in the sense that, the immediate births of each individual and its descendants are incorporated into the offspring distribution in the latter process. See Fig \ref{fig:equivalence} for a pictorial representation. The ``offspring'' of each individual in $\BP(F^*,h^*)$ correspond to the individuals in $\BP(F,h)$ comprising of:
\begin{enumeratea}
\item  Original children of this vertex (the first generation under the root in Fig. \ref{fig:equivalence}) with strictly positive lifetimes;
\item Starting with all descendent lines of this individual such that every individual on this descendant line has age zero when it dies (all grey vertices in the figure e.g. vertex $u$ in Figure \ref{fig:equivalence}), immediate children with positive lifetimes at the end points of these descendant lines.  An example is vertex $v$. 
\end{enumeratea}

\begin{figure}
\begin{tikzpicture}[scale=.8, 
  level 1/.style={sibling distance=30mm, level distance=20mm},
  level 2/.style={sibling distance=20mm},
  level 3/.style={sibling distance=15mm},
  level distance=20mm,
  red/.style={circle, inner sep=3pt, ball color=red!60},
  green/.style={circle, inner sep=3pt, ball color=green!80},
  blue/.style={circle, inner sep=3pt, ball color=lightgray!10},
  edge from parent/.style={very thick, draw},
  edge from parent path/.style={draw, very thick, edge from parent path},
  edge from parent/.append style={->}
]

\node[red] {$\rho$}
  child [edge from parent/.append style={draw=green}] {node [green] {}}
  child [edge from parent/.append style={draw=green}] {node [green] {}}
  child [edge from parent/.append style={draw=blue}] {node [blue] {}
    child [edge from parent/.append style={draw=green}] {node [green] {}}
    child [edge from parent/.append style={draw=green}] {node [green] {}}
  }
  child [edge from parent/.append style={draw=blue}] {node [blue] {}
    child [edge from parent/.append style={draw=blue}] {node [blue] {$u$}
      child [edge from parent/.append style={draw=green}] {node [green] {$v$}}
      child [edge from parent/.append style={draw=green}] {node [green] {}}
    }
    child [edge from parent/.append style={draw=blue}] {node [blue] {}}
  }
  child [edge from parent/.append style={draw=blue}] {node [blue] {}
    child [edge from parent/.append style={draw=blue}] {node [blue] {}
      child [edge from parent/.append style={draw=green}] {node [green] {}}
    }
    child [edge from parent/.append style={draw=blue}] {node [blue] {}
      child [edge from parent/.append style={draw=green}] {node [green] {}}
      child [edge from parent/.append style={draw=green}] {node [green] {}}
      child [edge from parent/.append style={draw=green}] {node [green] {}}
    }
  }
  child [edge from parent/.append style={draw=green}] {node [green] {}};

\end{tikzpicture}
\label{fig:equivalence}
\caption{Pictorial description of the equivalence between $\BP(F,h)$ and $\BP(F^*,h^*)$. The grey vertices represent descendants of the root (represented by red) with zero lifetimes while the green vertices represent their children (including children of the root) with positive lifetimes.  }
\end{figure}

Thus if {$B\sim h$} is the original offspring distribution (in the Figure, the number of offspring of the root $B_\rho = 6$), then going back to Figure \ref{fig:equivalence}, $B^*$ the offspring in $\BP(F^*,h^*)$ consists of all the green vertices ($B^*_\rho = 11$). Let $\set{N(t): t > 0}$ denote the number of individuals alive at time $t$ in ${\sf BP}(F,h)$. Then, by construction, this is {\bf the same} as the {\bf alive} population in the corresponding branching process ${\sf BP}(F^*,h^*)$. In comparison, the dead population of ${\sf BP}(F,h)$ at time $t$ is different from that of  ${\sf BP}(F^*,h^*)$.

We have now outlined the connection between ${\sf BP}(F,h)$ and  ${\sf BP}(F^*,h^*)$. Throughout the paper, we use $F^*, h^*$ to describe settings where the associated $\BP$ has strictly positive lifetimes. Definitions of various phenomena such as \textit{explosiveness} will be phrased in terms of the corresponding construction $\BP(F^*, h^*)$. Komj{\'a}thy in \cite{Kom} studies in great depth, conditions for explosiveness for age dependent branching process, and our setting arises as \cite[Example 4.1]{Kom}. Since ${\sf BP}(F^*,h^*)$ has $F^*(0)=0$, thus satisfies the assumption in \cite[Eqn 3.5]{Kom}.

Let $|{\sf BP}_t(F^*,h^*)|$ denote the population of dead individuals in ${\sf BP}(F^*,h^*)$ at time $t$ (in \cite{Kom} this is referred to as the already existing population at time $t$).

\begin{defn}[Explosive age-dependent process]
Say ${\sf BP}(F^*,h^*)$ is \emph{explosive} if there is a positive probability that $|{\sf BP}_t(F^*,h^*)|=\infty$. Otherwise call the branching process {\em conservative}.
\end{defn}

By Assumptions \ref{ass}, for the branching processes in this paper, $\pr({\sf BP}(F^*,h^*)  \text{ is explosive}) \in \set{0,1}$.  The next Definition describes general methodology in \cite{Kom} for necessary and sufficient conditions for explosiveness. 

\begin{defn}[Plump distributions, {\cite{Kom}}]
\label{def:plump}
Say that an offspring distribution $B^*$ is plump if there exists constants $C,m_0 \in (0, \infty)$ and $\delta\in (0,1)$ so that for all $m\geq m_0$, 
$$\pr(B^* \ge m) \ge \frac{C}{m^{1-\delta}}.$$ 
\end{defn}
\begin{lem}[Lemma 5.8 in \cite{Kom}]
Consider a branching process $\BP(F^*, h^*)$ with strictly positive lifetime distribution $F^*$. Suppose the offspring distribution $B^*\sim h^*$ is plump. Then the associated age-dependent branching process is explosive if and only if the birth-time distribution $F^*$ satisfies for some $\ep>0$,
\beq
\label{PL-exp}
\int_0^\ep [F^*]^{-1}( \exp(-1/y)) \frac{dy}{y} < \infty,\qquad \text{or equivalently} \qquad \int_{1/\ep}^\infty [F^*]^{-1}(e^{-u}) \frac{du}{u} < \infty.
\eeq

\end{lem}

\subsection{The comparison branching process}
\label{sec:comp-bp}
Next we connect the above setting to the random graph model with edge weights in Section \ref{sec:rg-cm}. This branching process will be used to understand asymptotics for the local structure of the shortest weight graph, emanating from a uniformly sampled vertex. Recall the original degree distribution $\vp$ and let $h$ denote the probability generating function of the size-biased distribution $\vq$ defined in \eqref{size_biased}. The phrase ``\emph{modified}'' in the next definition is used to remind the reader that the first generation has a different offspring distribution from subsequent generations. 

\begin{construction}[Modified age-dependent branching process]\label{comparison_BP}
Let $D\sim \vp$  and $\widetilde{\vB} = (\widetilde B_i)_{i\geq 1}$ be \emph{i.i.d.}~random variables following the size-biased distribution $\vq$, independent of $D$. Construct the following modified age-dependent branching process, denoted by $(\widetilde{\sf BP}_t)_{t\geq 0}=(\widetilde{\sf BP}_t(F,h))_{t\geq 0}$
:
\begin{enumeratei}
\item Start with the root $u_1 = \rho$ which dies immediately giving birth to $D$ offspring;
\item Each alive offspring has an independent life-length following the distribution $F$. We will label the vertices in increasing order of vertex death times,  breaking ties arbitrarily for zero lifetime vertices. 
\item When the $i$-th vertex dies, it gives birth to $\widetilde B_i$ alive offspring.  
\end{enumeratei}
In analogy with \eqref{birth_time} and \eqref{adapted_offspring}, define $(\widetilde{\sf BP}^*_t)_{t\geq 0}=(\widetilde{\sf BP}_t(F^*,h^*))_{t\geq 0}$ for the corresponding branching process of $(\widetilde{\sf BP}_t)_{t\geq 0}$ with strictly positive lifetime distribution.
\end{construction}

Before continuing, we setup a definition for use later. 

\begin{defn}[Discrete skeletons of $\widetilde{\BP}$ and $ \widetilde{\BP}^*$]
\label{def:disc-skeleton}
Write $(\widetilde{\BP}[m])_{m\geq 1} = (\widetilde{\BP}_{T_m})_{m\geq 1}$ viewed as a sequence of increasing trees with vertex labels according to the order of birth and edge labels corresponding to the life-lengths and call this the discrete skeleton of $\widetilde{\BP}$. Analogously, let $(\widetilde{\BP}^*[m])_{m\geq 1}$ denote the corresponding discrete skeleton for $(\widetilde{\sf BP}^*_t)_{t\geq 0}$. 

\end{defn}

In order to understand the first passage percolation model, the starting point  is the zero-weight cluster containing the root vertex, which corresponds to the dead population in $\widetilde{\sf BP}_{0+}(F,h)$. In addition, we want to understand the distribution of the number of positive-weight edges attached to the vertices in the zero-weight cluster of the root. This corresponds to the number of alive individuals in  $\widetilde{\sf BP}_{0+}(F,h)$, which is the same as in $\widetilde{\sf BP}_{0+}(F^*,h^*)$.

The inhomogeneity of the root offspring distribution in contrast with the rest of the process is a minor technical annoyance. To simplify the question, we start by considering the homogeneous age-dependent branching process, denoted by ${\sf BP}=({\sf BP}_t)_{t\geq 0}$, whose offspring distribution is given by $\vq$ as in \eqref{size_biased} and the birth times follow the same distribution $F$ as for the original process $(\widetilde{\sf BP}_t)_{t\geq 0}$. Assuming this can be analyzed,  suppose the root in $(\widetilde{\sf BP}_t)_{t\geq 0}$ has degree $D\sim \vp$. Let $D^{(0)}= \mathrm{Binomial}(D,p_c)$ denote the number of edges of the root having zero weight. Let $\{{\sf BP}^{(i)}: i\geq 1\}$ be independent copies of the homogeneous branching process ${\sf BP}$ and use $|{\sf BP}^{(i)}_t|$ to denote the size of the dead population in ${\sf BP}^{(i)}_t$. Then,
$$|\widetilde{\sf BP}_{0+}|\overset{d}{=} 1+ \sum_{i=1}^{D^{(0)}} |{\sf BP}^{(i)}_{0+}|.$$
This allows one to read off properties for $\widetilde{\sf BP}$ easily from ${\sf BP}$.  Hence we will focus on the {\bf size of a zero-weight cluster} in ${\sf BP}$. This can be studied via an exploration random walk which recursively produces two sets, the set of \emph{active} vertices $(\cA_i : i \geq 1)$ and the set of \emph{removed} vertices $(\cR_i : i \geq 1)$, as follows:
\begin{enumeratea}
\item At step $i=0$, set $\cA_0 = \set{\rho}$ (namely the root as the initial active vertex $\rho = v_0$) and initialize the removed set $\cR_0 = \emptyset$. 
\item At every step $i\geq 1$, we randomly pick an active vertex, denoted by $v_i$, and reveal the number of its \emph{active} offspring $\eta_i=\mathrm{Binomial}(B_i,p_c)$, where $B_i$ is an independent random variable following the size-biased distribution $\vq$ as in \eqref{size_biased}. The $\eta_i$ offspring of $v_i$ are added to the active set $\cA_i$ which is updated to $\cA_{i+1}$.  Vertex $v_i$ is added to the removed set. Update the removed set via $\cR_{i+1} = \cR_i \cup \set{v_i}$. 
\item Repeat (b) until there are no more active vertices.
\end{enumeratea}

For $i\geq 0$, let $Q_i = |\cA_i|$ denote the number of active vertices at step $i$, initialized at $Q_0 =1$.  Since $(B_i)_{i\geq 0}$ is a collection of \emph{i.i.d.}~random variables following the size-biased distribution $\vq$, the dynamics of the size of the active set $Q_i = |\cA_i|$ can be written as a random walk with \emph{i.i.d.}~increments, 
\beq\label{RW}
Q_i=Q_{i-1}+\eta_i-1 \quad \text{ where } Q_0=1 \text{ and } \eta_i= \mathrm{Binomial}(B_i,p_c).
\eeq
At time $\chi=\min\{ i\geq 0: Q_i=0\}$ the exploration terminates and the size of the zero-weight cluster is given by $\chi$. 
To compute the distribution of $\chi$, we will use a result from \cite{Pitman} originating in \cite[Lemma 6.1]{Tak}.  
\begin{lem}[Kemperman's formula]\label{Kemperman}
 Suppose that for any $m\geq 1$, $ (\eta_1, \ldots \eta_m)$ are cyclically exchangeable integer valued random variables with $\eta_i\geq -1$ a.s. Define $W_j = \sum_{i=1}^j \eta_i$ and
$$
T_{-k} = \inf\{ j > 0 : W_j-j = -k\}.
$$
Then, 
$$
\pr( T_{-k}=m) = \frac{k}{m} \pr(W_m-m=-k).
$$
\end{lem}

We now describe how this leads, in the critical regime, to the scaling of the size of the active set of vertices of the exploration process. 

\begin{lem}[Size of the zero-weight cluster]\label{size_cluster}
Let $\chi=\min\{ i\geq 0: Q_i=0\}$ for the random walk $Q_i$ defined in \eqref{RW}. There exist strictly positive finite constants $C_1,C_2$ so that for all $m \geq 1$, 
\beq
C_1m^{-1/2}\leq \pr(\chi\geq  m) \leq C_2m^{-1/2}.
\label{tauasy}
\eeq

\end{lem}

\begin{rem}
\label{rem:crit}
The above Lemma \ref{size_cluster} and Lemma \ref{cluster_degree} below form the main ingredients of the proof and are not effected if the probability of zero weight edges is changed from $p_c$ to a different point $p_c+\lambda/n^{1/3}$ in the critical scaling window of $\CM_n(\vp)$. Thus the results on the asymptotics for passage times between typical vertices seem robust to such perturbations. 
\end{rem}

\begin{proof}
It follows from \eqref{RW} and the definition of $W_j$ that $Q_j = 1+W_j-j$ for $j\geq 0$. Note that $(\eta_i)_{i\geq 1}$ are \emph{i.i.d.}~lattice distributions with span one,  with mean 1 and variance $\var_\eta = (1-p_c) + p_c^2 \var(B_i)< \infty$ by Assumption \ref{ass}. Applying the local central limit theorem (see e.g., \cite[Theorem 3.5.3]{PTE5}) to $R_m-m$ we get that 
$$
\lim_{m\to\infty} m^{1/2} \pr(W_m-m =-1)= (2\pi\var_\eta)^{-1/2}.
$$
We are interested in $\chi = T_{-1}$, whose law, according to Lemma \ref{Kemperman} is,  
$$\pr( \chi=m) = \frac{1}{m} \pr(W_m-m=-1).$$
Hence there exists some $m_0\in \mathbb{N}$ such that for $m\geq m_0$,
$$\frac{1}{2}(2\pi\var_\eta)^{-1/2}m^{-3/2}\leq \pr(\chi=m)\leq 2(2\pi\var_\eta)^{-1/2}m^{-3/2}.$$ 
\end{proof}

Next we consider the number of positive-weight edges connected to a zero-weight cluster of the root (as described above, this is the number of alive individuals in  $\widetilde{\sf BP}_{0+}(F,h)$, which is the same as in $\widetilde{\sf BP}_{0+}(F^*,h^*)$). Recall that $B_i$ denotes the offspring degree of the vertex $v_i$ chosen at the $i$-th step in the exploration process. For $i\geq 1$, let $\xi_i=\mathrm{Binomial}(B_i,1-p_c)=B_i-\eta_i$ denote the number of positive-weight edges attached to $v_i$. Define 
$$S_n =\sum_{i=1}^n \xi_i, \quad W_n=\sum_{i=1}^n \eta_i \quad \text{and}\quad V_n = \sum_{i=1}^n B_i.$$
We will bring in the restriction of viewing the above sequences only up to the stopping time $\chi$ in a subsequent Lemma. For the time being view the sequences $\set{S_n:n\geq 1}$, $\set{W_n:n\geq 1}$ and $\set{V_n:n\geq 1} $ as a collection of coupled random walks driven by the sequence $\set{B_i:i\geq 1}$ and additional Binomial sampling.  Note that $V_n=W_n+S_n$ and $W_n=\mathrm{Binomial}(V_n, p_c)$, $S_n=\mathrm{Binomial}(V_n, 1-p_c)$. According to the previous discussion, $S_\chi$ then gives the number of positive-weight edges connected to a zero-weight cluster.
 In order to understand $S_{\chi}$, we first start by exploiting the connection between $S_n$ and $W_n$.
 \begin{lem}\label{R&S}
 There exists some constants $\rho^+,\rho^-,c>0$ such that $ \rho^-W_n< S_n<\rho^+W_n$ with probability at least $1-2\exp(-2cn)$.
 \end{lem}
 \begin{proof}
 Since $W_n=\mathrm{Binomial}(V_n, p_c)$ and $S_n=\mathrm{Binomial}(V_n, 1-p_c)$, by standard large deviations for the Binomial distribution, for any $\ep>0$ there exist constants $c_1,c_2>0$ such that,
\begin{align*}
&\pr( |W_n - p_cV_n| \ge \ep p_c V_n ) \le \exp(-c_1n), \\
&\pr( |S_n - (1-p_c)V_n| \ge \ep(1-p_c) V_n ) \le \exp(-c_2n).
\end{align*} 
 Here we use the fact  $V_n\geq n$ to replace $V_n$ by $n$ in the exponent. Algebra then yields,
 \begin{align*}
  \pr\left( \bigg|\frac{S_n}{(1-p_c)V_n}-\frac{W_n}{p_c V_n}\bigg|\geq 2\ep\right) &\leq   \exp(-c_1n)+\exp(-c_2n)\leq 2\exp(-\min\{c_1,c_2\}n).
 \end{align*}
 Thus with probability at least $1-2\exp(-\min\{c_1,c_2\}n)$ one has,
 $$  \frac{1-p_c}{p_c}W_n-2\ep(1-p_c)V_n\leq S_n\leq  \frac{1-p_c}{p_c}W_n+2\ep(1-p_c)V_n.$$
 The above rearranges to,
 $$  \frac{(1-p_c)( 1/p_c-2\ep)}{1+2\ep(1-p_c)}W_n\leq S_n\leq \frac{ (1-p_c)(1/p_c+2\ep)}{1-2\ep(1-p_c)}W_n,$$
 as $V_n=W_n+S_n$. Choosing $\ep>0$ sufficiently small and taking $\rho^-<\frac{ (1-p_c)( 1/p_c-2\ep)}{1+2\ep(1-p_c)}, \rho^+> \frac{ (1-p_c)(1/p_c+2\ep)}{1-2\ep(1-p_c)}$ completes the proof.
 \end{proof}
 Now consider the number of positive weighted edges $S_\chi$, originating from a zero-weight cluster.

 \begin{lem}[Number of positive-weight edges]\label{cluster_degree}
 There exist $C_1,C_2>0$ such that for $m\geq 1$,
  $$C_1m^{-1/2}\leq \pr( S_\chi \geq m) \leq C_2m^{-1/2}.$$
 \end{lem}
 \begin{proof}
 By definition we have $W_\chi=\chi-1$. Let $\rho^{\pm},c$ be the constants in Lemma \ref{R&S}. As $S_n$ and $W_n$ are both non-decreasing with respect to $n$, 
 \begin{align*}
 \pr( S_\chi \geq \rho^- m)&\geq \pr( S_{\chi}\geq \rho^- m, \chi>m)\geq \pr( S_m \geq \rho^- m, \chi>m )\\
 &\geq \pr( S_m \geq \rho^- W_{m}, W_m \geq  m,  \chi>m )\\
 &\geq \pr( W_m \geq  m,  \chi>m )-2\exp(-cm)\\
 &=\pr(  \chi>m)-2\exp(-cm)\geq Cm^{-1/2},
 \end{align*}
 for some $C>0$, where the last equality is because on the event $\{\chi>m\}$ we have $W_m-m>-1$, i.e., $W_m\geq m$. The lower bound then follows from Lemma \ref{size_cluster}.
 
 For the upper bound we use the fact that $W_\chi=\chi-1$ and observe,
 \begin{align*}
 \pr(S_\chi\geq \rho^+m)&\leq \pr(S_\chi\geq \rho^+m, \chi>m)+\pr(S_\chi\geq \rho^+m, \chi\leq m)\\
 &\leq \pr(\chi>m)+\pr(S_\chi\geq  \rho^+W_\chi, \chi\leq m)\\
 &\leq Cm^{-1/2}+m\exp(-2cm)\leq C'm^{-1/2},
 \end{align*}
 for some $C'>0$.
 
 \end{proof}

\subsection{Smallest-weight graph in configuration model}
\label{sec:shortest-wg}

As described in Section \ref{sec:proof-overview}, to understand the optimal path between the vertices $u\neq v \in [n]$, we  grow the shortest weight clusters between these two vertices simultaneously till they have a common vertex. The goal of this Section is to describe the approximation of such clusters by the branching process described above. We first need a few constructions wherein we simultaneously explore the shortest weight graph around a vertex and construct the graph corresponding to this cluster. Recall the construction of the configuration model in Section \ref{sec:rg-cm} including the notion of free half-edges/stubs and the matching procedure to create a full edge. We write $(\SWG^{v}_m)_{m\geq 1}$ for \emph{smallest-weight graph}  starting from vertex $v$ constructed sequentially described next. We sometimes use the phrase ``\emph{flow clusters}'' in place of shortest weight graphs, as this is more evoctive of the underlying process, namely flow starting from the source vertex and moving at rate one through the graph using the edges with prescribed lengths.

\begin{defn}[Smallest-weight graph]\label{SWG}
 Initialize with $\SWG_{1}^{v}=\{v\}$. Add the edge and vertex with minimal edge weight connecting $v$ to one of its neighbors. The growth of $(\SWG^{v}_m)_{m\geq 1}$ is then defined recursively: given $\SWG^v_{m}$, we obtain $\SWG^v_{m+1}$ by adding the edge and the end vertex connected to $\SWG^v_{m}$ with minimal edge weight (ties broken arbitrarily). 
\end{defn}

With some abuse of notation we will write $(\SWG^{v}_m)_{m\geq 1}$ as a collection of vertices while implicitly including also the edges. The smallest-weight graph (SWG) in principle can be constructed on any given graph with edge weights. When the underlying graph is described by the configuration model $\CM$, we next describe the approach in \cite{BHH10b} to simultaneously grow the graph as well as the shortest-weight graph from a given vertex. In the construction below, a stub that has not been paired is referred to as a \emph{free} stub. We will assign weights to stubs according to the distribution $F$. A stub with weight zero is said to be a \emph{zero stub} and otherwise a \emph{positive stub}. A stub whose weight has not yet been assigned is said to be \emph{unmarked}. The construction proceeds conditional on the degree sequence $\vd:= \set{d_i: i\in [n]}$.

\begin{construction}\label{SWG_on_CM}
The construction of the shortest weight graph starting from vertex $v$,  $(\SWG^{v}_m)_{m\geq 1}$, on the configuration model proceeds iteratively:
\begin{enumeraten}
\item Start with $\SWG_{1}^{v}=\{v\}$, where the vertex $v$ has $d_v$ stubs, each with an independent weight drawn from distribution $F$.
\item Given $\SWG^v_m$ with known stub weights, we select the stub with the minimal weight (breaking ties arbitrarily) attached to $\SWG^v_m$ and pair it with a free stub chosen uniformly at random, designating its endpoint as $v_{m+1}$:

 \begin{itemize}
\item If the chosen stub is unmarked, we have obtained $\SWG^v_{m+1}=\SWG^v_m\cup \{v_{m+1}\}$. We then independently assign weights to each remaining free stub of $v_{m+1}$ that was previously unmarked.

\item  If the chosen stub is marked, we say a \textbf{collision} occurs, and the construction stops.
\end{itemize}
\item Repeat Step (2) until a collision occurs.
\end{enumeraten}

\end{construction}

To simplify notation, when there is no scope for confusion,  we use $\SWG$ to denote the smallest-weight graph starting from a uniformly chosen vertex in the given graph.

\subsection{Coupling SWG with age-dependent branching process}
\label{sec:proof-coupling}
The goal of this Section is to summarize both known results as well as derive technical estimates that will allow us to couple, local neighborhoods of the $\SWG$ with a corresponding age dependent branching process.  Obstacles in coupling include:

\begin{enumeratea}
\item Arising of collision events. Since we grow two flow clusters, say first starting at $v$ and then $u$, one obvious example of collision is that the flow cluster from $v$ finds vertex $u$. 
\item The fact that for finite $n$, the outward exploration from a typical vertex gives rise to forward degrees that is only approximately independent and only approximately has the limit size biased distribution $\vq$. 
\end{enumeratea}

Growing shortest weight graphs from two vertices on the configuration model, either in the first passage percolation setup \cites{BHH10b,BHH17}, or in trying to understand just graph distance \cites{van2005distances,10.1214/EJP.v12-420}, is now relatively standard and all of the above obstacles have been tackled in previous papers.  The aim of this Section is to review the technical components necessary for addressing these approximations, but first, we will establish some notation.

Recall that we used $\CM_n(\vp)$ for the random graph generated by first sampling the degree sequence $\vD = (d_i)_{i\in [n]}$ in an \emph{i.i.d.}~fashion from $\vp$ and then considering the configuration model and the ensuing first passage percolation problem on the model. We will typically work conditionally on the generated degree sequence $\vD$ and will ocassionally write $\CM_n(\vD)$ and write $\pr_n$ for computing probabilities conditional on the degree sequence to make this distinction.

Now recall the Construction \ref{SWG_on_CM}. For $m\geq 1$, define
\beq\label{real_time}
R_m:=\min\{ j\geq 1: \SWG_j \text{ contains $m$ vertices}\}.
\eeq
Note that we have $R_m=m$ exactly when $\SWG$ is a tree, i.e., no cycle has been created in the construction by time $m$. 
We say a \emph{collision} has occurred when a cycle is created during the growth of $\SWG$, and define the collision time to be 
\beq\label{collision_time}
R_{\mathrm{col}}=\min\{ m\geq 1: R_m>m  \}.
\eeq

Construction \ref{SWG_on_CM} can be extended to take into account collisions as in \cite[Section 4.2]{BHH10b} where, a modified construction of the SWG was introduced where edges that create cycles (and self-loops) are deleted and replaced by ``artificial stubs". As a result, the modified SWG constructed on a graph remains a tree that can essentially be coupled with (and stochastically dominated by) a branching process whose offspring is driven by the size-biased version of the empirical distribution of degrees. Our arguments for understanding the passage time need the construction only before this collision time so these additional approximations will not be required.  We observe that $(\SWG_m)_{m\geq 1}$ remains a tree until the first collision and focus on the growth of the SWG before $R_{\mathrm{col}}$.

  Let $(\SWG^v_m)_{m\geq 1}$ denote the SWG on $\CM_n(\vD)$ starting from a uniformly chosen vertex $v$ as described in Construction \ref{SWG_on_CM}. 
Let $v_i$ denote the vertex explored in the $i$-th step in $(\SWG^v_m)_{m\geq 1}$. For $i\geq 2$, we will denote by $B_i$ the degree of vertex $v_i$ minus 1, i.e., $B_i$ is the \emph{forward degree} of $v_i$. We are interested in the forward degrees of previously unexplored vertices, i.e., $(B_{R_m})_{m\geq 2}$.  Now consider two distinct and uniformly sampled vertices $u,v$. First grow the shortest weight graph $\SWG^v_{m_n}$ for up to $m_n$ steps (unless a collision event happens) and then start growing the shortest weight graph from $\SWG^u$ unless $u \in \SWG^v_{m_n}$ or one of the stubs previously found in $\SWG_{m_n}^v$ is sampled. Define $R^v = R_{\mathrm{col}}^v$ as above and after growing $\SWG_{m_n}^v$, define $\tilde{R}^u$ to be the first time $m\geq 0$ such that either $u$ is already in $\SWG_{m_n}^v$ or another active stub in $\SWG_{m_n}^v$ is sampled by $\SWG^u$ or there is a collision in the growth of the cluster $\SWG^u$. We use notation $\tilde{R}^u$ since, owing to the sequential construction, first starting from $v$ and then $u$,  it is different from $R^v$ in distribution.

\begin{prop}[Coupling over mesoscopic growth of the clusters]\label{couple_iid}
First consider the shortest weight graph $\SWG^v$ from a single randomly sampled vertex $v$. Fix any $m_n = o(\sqrt{n})$.  Then under Assumptions \ref{ass}:
\begin{enumeratea}
\item {\bf Probability of collision from a single source:} Then the probability of collision $\pr( R_{\mathrm{col}}\leq m_n) \to 0$ as $n\to \infty$. 
\item {\bf Coupling to an {\it i.i.d.} sequence:} Let $\vd$ be a given degree sequence and let $L_n = \sum_{i=1}^n d_i$. Consider the empirical distribution
\beq\label{empirical_dist}
 g^{(n)}_k :=\sum_{i=1}^n \frac{ (k+1)\ind\{d_i=k+1\}}{L_n}, \qquad k\geq 0.
\eeq
Let $\set{Y_i^{\sss(n)}:i\geq 2}$ be an \emph{i.i.d.}~sequence with distribution $(g^{(n)}_k)_{k\geq 0}$.  There is a coupling of $(B_{R_i})_{2\leq i\leq m_n }$ with $(Y_i)_{2\leq i\leq m_n}$ such that 
\[\pr(\cap_{i=2}^{m_n}\set{B_{R_i} = Y_i^{\sss(n)}}) \to 1, \qquad \text{ as } n\to\infty. \]
\item {\bf Collision events from multiple sources:} Next consider the growth from two clusters first from $v$ and then from $u$ as described above. Then $\pr( \min (R^v, \tilde{R}^u) \leq m_n) \to 0$ as $n\to \infty$. 
\item {\bf Coupling for multiple sources:} The coupling in (b) with an \emph{i.i.d.}~sequence for a single source $\SWG^v$ continues to hold starting from two distinct randomly sampled sources.  
\end{enumeratea}

\end{prop}

\begin{proof}
\begin{enumeratea}
Part(a) and (c) follows from the statement and proof of \cite[Lemma A.1]{BHH10b}. 
For part(b) and (d) see \cite[Lemma 5.1 and Eqn. (5.11)]{BHH17} and its proof. 
\end{enumeratea}

\end{proof}

\begin{rem}
    The above proposition allows us to describe various simultaneous couplings of the shortest weight graphs from two randomly sampled vertices by essentially proving the desired result for one of the shortest weight graphs, since as long as one grows these clusters up to a ``mesoscopic'' time scale where the total number of half-edges used is $o_{\pr}(\sqrt{n})$ one has ``independent'' evolution of the two shortest weight graphs. This theme will be present in many of the remaining proofs of this Section where we mainly describe the proof only for one cluster. 
\end{rem}

In principle,  we would like to couple the forward degrees to an {\it i.i.d.} sequence with distribution given by the ``limit'' size biased distribution $\vq$ in \eqref{size_biased} as opposed to the empirical size-biased pmf as in \eqref{empirical_dist}.  The next result is the initial step on this path and follows similar lines as \cite[Proposition A.1.1]{distfv}.

Recall that $\vD=(D_1,\dots,D_n)$ denotes a sequence of \emph{i.i.d.}~samples from $\vp$.
\begin{lem}\label{tv_bound}
Let  $g^{(n)}=(g^{(n)}_k)_{k\geq 0}$ be defined as in \eqref{empirical_dist} with $\vd=\vD$.
When the degree exponent $\tau$ of the original pmf $\vp$ satisfies $\tau>3$, for any $\ep>0$ we have
$$\pr( d_{\mathrm{TV}}(g^{(n)},g)\geq n^{-(\frac{\tau-2}{2\tau-3}-\ep)})=o(1).$$
\end{lem}
\begin{proof}
Fix $a,b,\alpha>0$. First observe that
\begin{align}\label{total_variation}
\nonumber \sum_{k=0}^\infty |g_k^{(n)}-g_k|&\leq \sum_{k=0}^\infty \left(|g^{(n)}_k-\frac{n\mu}{L_n}g_k|+g_k|\frac{n\mu}{L_n}-1|\right)\\
&=  \sum_{k=0}^\infty \frac{(k+1)}{L_n} | \sum_{i=1}^n(\ind\{D_i=k+1\}-p_{k+1})|+|\frac{n\mu}{L_n}-1|.
\end{align}
To bound \eqref{total_variation}, define the event
\begin{align}
\nonumber F_n&=\{ |\frac{n\mu}{L_n}-1|\leq n^{-\alpha}\} \cap \left\{ \frac{1}{n} \sum_{i=1}^n (D_i+1)\ind\{ D_i\geq n^a\}\leq n^{-b}\right\}\\
&\quad \cap\left\{ \frac{1}{n}\sum_{k=0}^{n^a} (k+1) |\sum_{i=1}^n(\ind\{D_i=k+1\}-p_{k+1})|\leq n^{-b}\right\}=: F_{1,n}\cap F_{2,n}\cap F_{3,n}
\end{align}
so that on $F_n$ one has 
$$\sum_{k=0}^\infty  |g_k^{(n)}-g_k|\leq n^{-\alpha}+\frac{2(1+n^{-\alpha})}{\mu}n^{-b},$$
where the second term comes from the inequality $\frac{\mu n}{L_n}\leq 1+n^{-\alpha}$ guaranteed on the event $F_n$. It remains to find suitable values of $a,b,\alpha$ such that $\pr(F_n^c)=o(1)$. It follows directly from the central limit theorem for $L_n$ that for any choice of $0<\alpha<1/2$, $\pr(F_{1,n}^c) = o(1)$. Next, an application of Markov's inequality gives for some constant $C'\in \bR_+$,
$$\pr(F_{2,n}^c)\leq n^b\E[(D_1+1)\ind\{D_1\geq n^a\}]\leq C'n^{b-a(\tau-2)}.$$
Finally, using Markov inequality and Cauchy-Schwarz inequality,
\begin{align*}
\pr(F_{3,n}^c)&\leq n^{1-b} \E\left[ \sum_{k=0}^{n^a} (k+1) |\sum_{i=1}^n (\ind\{D_i=k+1\}-p_{k+1})|\right]\\
&\leq n^{b-1} (n^a+1)^{1/2} \E\left[ \left(\sum_{k=0}^{n^a} (k+1)^2(\sum_{i=1}^n (\1\{D_i=k+1\}-p_{k+1})^2\right)^{1/2}\right]\\
&\leq 2n^{b+a/2-1}\left( \sum_{k=0}^{n^a} (k+1)^2 \E(\sum_{i=1}^n (\1\{D_i=k+1\}-p_{k+1})^2\right)^{1/2},
\end{align*}
where the last line follows from Jensen's inequality. Next note that, 
$$\E(\sum_{i=1}^n (\1\{D_i=k+1\}-p_{k+1})^2=\mathrm{Var}\left( \sum_{i=1}^n (\1\{D_i=k+1\}\right)=np_{k+1}(1-p_{k+1})\leq np_{k+1}.$$
Hence,
\begin{align*}
\pr(F_3^c)&\leq 2n^{b+a/2-1}\left( \sum_{k=0}^{n^a} (k+1)^2 np_{k+1}\right)^{1/2}\leq C n^{b+a/2-1/2}n^{a\max\{0, 3-\tau\}/2}.
\end{align*}
As we are assuming $\tau>3$, it suffices to choose $a\in (0,1)$, $0<b<\min\{ a(\tau-2), (1-a)/2\}$ and $\alpha \in (0,1/2)$ to ensure $\pr(F^c)=o(1)$. The bound is optimized when $a=1/(2\tau-3)$, where for any $\ep>0$ we can choose $b=\frac{\tau-2}{2\tau-3}-\ep$ and $\alpha\in (b,1/2)$ so that 
$$\pr( d_{\mathrm{TV}}(g^{(n)},g)\geq n^{-(\frac{\tau-2}{2\tau-3}-\ep)})=o(1).$$
\end{proof}

\begin{prop}\label{coupling_forward_degrees}
For $(\SWG^v_m)_{m\geq 1}$ on $\mathrm{CM}_n(\vD)$ as in the setting of Proposition \ref{couple_iid} with $\tau>3$, for any $\rho<\frac{\tau-2}{2\tau-3}$ the random vector of forward degrees $(B_{i})_{i=2}^{n^{\rho}}$ can be coupled to an independent sequence of random variables $(B^{(ind)}_{i})_{i=2}^{n^{\rho}}$ with probability mass function $g = \vq$ given in \eqref{size_biased} in the sense that $(B_{i})_{i=2}^{n^{\rho}}=(B^{(ind)}_{i})_{i=2}^{n^{\rho}}$ whp. 

The same result continues to hold for two sequentially constructed shortest weight graphs $\SWG^v_{n^{\rho}}, \SWG^u_{n^{\rho}}$ starting from two randomly sampled vertices $u,v$ using disjoint collection of independent random variables for each shortest weight graph and such that the shortest weight graphs are disjoint. 
\end{prop}

\begin{proof}
We will describe the proof for one shortest weight graph; for two shortest weight graphs, one uses the same proof and Proposition \ref{couple_iid} (c) and (d). Let  $g^{(n)}=(g^{(n)}_k)_{k\geq 0}$ be defined as in \eqref{empirical_dist} with $\vd=\vD$.
As $\rho<1/2$ by definition, it follows from Lemma  \ref{couple_iid} that there exists an \emph{i.i.d.}~sequence  $\set{Y_i^{\sss(n)}: 2\leq i\leq n^{\rho}}$ with distribution $g^{(n)}$ such that the coupling $(B_{i})_{i=2}^{n^{\rho}}=(Y^{(n)}_i)_{2\leq i\leq n^{\rho}}$ holds whp.

It remains to show that the coupling $(Y^{(n)}_i)_{i=2}^{n^{\rho}}=(B^{(ind)}_{i})_{i=2}^{n^{\rho}}$ holds whp. For any $\rho<\frac{\tau-2}{2\tau-3}$, fix $\ep=(\frac{\tau-2}{2\tau-3}-\rho)/2$ and it follows from Lemma \ref{tv_bound} that
$$\pr( d_{\mathrm{TV}}(g^{(n)},g)\geq n^{-(\frac{\tau-2}{2\tau-3}-\ep)})=o(1).$$
That is, whp the realization of $\vD$ satisfies $d_{\mathrm{TV}}(g^{(n)},g)\leq n^{-(\frac{\tau-2}{2\tau-3}-\ep)}$. Given such $\vD$, we can apply a simple union bound to see that the coupling fails with probability at most
\begin{align*}
\pr\left( (Y^{(n)}_i)_{i=2}^{ n^{\rho}}\neq (B^{(ind)}_{i})_{i=2}^{n^{\rho}}\right)&\leq n^\rho d_{\mathrm{TV}}(g^{(n)},g)\leq n^{-\ep}.
\end{align*}
Therefore, the coupling $(B_{i})_{i=2}^{n^{\rho}}=(B^{(ind)}_{i})_{i=2}^{n^{\rho}}$ holds whp for any $\rho<\frac{\tau-2}{2\tau-3}$.

\end{proof}

We shall now state the main coupling result between the shortest weight graph and the associated age dependent branching process. We once again initially describe the coupling started from one randomly sampled vertex $v$ and then extending the coupling to two shortest weight clusters.  To summarize the main coupling results, we will need some notation. Recall that $\cG_n \sim \CM_n(\vp)$ denotes the underlying graph and $\mathcal{E}(\cG_n)$ denotes the corresponding edge set. Next, recall that $\{t_e: e\in \mathcal{E}(\cG_n)\}$ denoted the passage times across edges while $\set{T_n(u^\prime,v^\prime): u^\prime,v^\prime\in [n]}$ denoted the minimal inter-vertex passage times. For later use, the above easily extends to minimal passage times between sets: for $A, B \subseteq [n]$,  $T_n(A,B) = \min_{u\in A, v\in B} T_n(u,v)$.  Also recall the discrete skeleton from Def. \ref{SWG}.  To couple the SWG starting from vertex $v$ with this continuous time branching process, we will look at the subsequence $(\SWG^v_{\sigma_\ell})_{\ell\geq 1}$ defined recursively as follows. Let 
$$\mathcal{C}(v):=\{ u^\prime \in [n]: T_n{(u^\prime,v)}=0\},$$
denote the \emph{zero-weight cluster} containing the intial  vertex $v$, which consists of all the vertices that can be reached from $v$ \emph{instantaneously}. Recall from Construction \ref{SWG_on_CM} that for $m\geq 1$, $v_{m}$ denotes the vertex we add at the $m$-th step in the growth of $(\SWG^v_m)_{m\geq 1}$, i.e., $\SWG^v_{m}=\SWG^v_{m-1}\cup \{v_{m}\}$ for $m\geq 1$, where we write $v=v_1$ and $\SWG^v_0=\emptyset$ to unify notation. Next for $m\geq 2$ let $\cC(v_m)$ consist of the zero-weight cluster found by the shortest weight-graph (consisting of all vertices not yet included in the exploration process) when exploring from $v_m$ (after having finished exploration from $v_1, v_2, \ldots, v_{m-1}$).

Since we are eventually interested in connecting the above exploration to explosion times (or lack thereof) of the associated branching process $\widetilde\BP^*$ (with strictly positive life-lengths, see Section \ref{sec:adbp}), we will need to keep track of the subsequence of stopping times when $(\SWG^v_m)_{m\geq 1}$ only has stubs of positive weights emanating from it,  written as $(\sigma_\ell)_{\ell\geq 0}$ and constructed as follows. 
Let $\sigma_0=0$, and recursively define 
\beq\label{SWG_subseq}
\sigma_{\ell+1}=\sigma_\ell+|\mathcal{C}(v_{\sigma_\ell+1})| \quad \text{ for }\ell\geq 0.
\eeq
By this definition, $(\sigma_\ell)_{\ell\geq 1}$ are the times when we finish adding the vertices in a new zero-weight cluster to the SWG, and hence $(\SWG^v_{\sigma_\ell})_{\ell\geq 1}$  is a sequence of shortest weight graphs where every stub emanating from the graphs has strictly positive weight.

Recall the construction of the branching process in Construction \ref{comparison_BP} and the corresponding skeleton in Definition \ref{def:disc-skeleton} where the labelling is in the order of birth via (with $u_1$ denoting the root) as, $(u_1, u_2, \ldots, )$. Using the same notation as for the branching process let $\cC^{\BP}(u_i)$ denote the zero-weight cluster of $u_i$ namely the collection of all descendants that are born instantaneously. Analogous to \eqref{SWG_subseq} we can define a sequence of stopping times $(\Sigma_\ell)_{\ell\geq 0}$ for the branching process $\widetilde \BP$ with $\Sigma_0=0$ and
\beq\label{BP_subseq}
{\Sigma}_{\ell+1} ={\Sigma}_\ell+|\mathcal{C}^{\BP}(u_{\Sigma_\ell+1})| \quad \text{ for }\ell\geq 0.
\eeq
Note that for $\ell\geq 1$ the distribution of $|\mathcal{C}^{\BP}(u_{\Sigma_\ell+1})|$ is given by $\chi$ in Lemma \ref{size_cluster} since the sizes of these clusters involve \emph{i.i.d.}~forward degrees $(B_i)_{i\geq 2}$.

 Let $\theta_1=0$ and define recursively for $\ell\geq 1$,
\beq\label{stoppingtime_BP}
\theta_{\ell+1}=\inf\{ t>\theta_\ell: |\widetilde{\sf BP}(F,h)_{t}|>|\widetilde{\sf BP}(F,h)_{\theta_\ell}|\}.
\eeq
{We use these stopping times $(\theta_\ell)_{\ell\geq 0}$ to track the births that are not instantaneous.} Note that although the dead populations are different in $\widetilde{\sf BP}_{t}(F,h)$ and 
$\widetilde{\sf BP}_{t}(F^*,h^*)$, $(\theta_\ell)_{\ell \geq 0}$ can be equivalently defined with respect to $\widetilde{\sf BP}_{t}(F^*,h^*)$,
$$
\theta_{\ell+1}=\inf\{ t>\theta_\ell: |\widetilde{\sf BP}_t(F^*,h^*)|>|\widetilde{\sf BP}_{\theta_\ell}(F^*,h^*)|\}.
$$
Let 
\beq\label{explosion_time}
\theta_\infty:=\lim_{\ell\to\infty} \theta_\ell \in (0,\infty]
\eeq
denote the time for eventual completition of the branching process $\widetilde{\sf BP}_{t}(F^*,h^*)$ with $\theta_\infty < \infty$ {\bf if and only if} the branching process explodes within finite time.

\begin{prop}[Coupling between SWG and BP]\label{coupling_SWG_BP}
Fix any $\rho<\frac{\tau-2}{2\tau-3}$. First consider the shortest weight graph $\SWG^v$ from one randomly sampled vertex $v$.  
\begin{enumeratea}
\item There exists a coupling of the shortest weight graph process $\SWG$ with the discrete skeleton of $\widetilde{\BP}$ such that with high probability as $n\to\infty$,
 $$(\SWG^v_i)_{i\leq n^{\rho}}=(\widetilde{\BP}[i])_{i\leq n^{\rho}}.$$
 \item For any sequence $(\ell_n)_{n\geq 1}$ satisfying $\ell_n/n^{\rho/2}\to 0$, with high probability the stopping time $\sigma_{\ell_n}\ll n^{\rho}$. In particular there is a coupling of the subtree spanned by strictly positive life-lengths in $\SWG$ such that 
$$(\SWG^v_{i})_{1\leq i\leq \sigma_{\ell_n}}= (\widetilde{\sf BP}[i])_{1\leq i\leq \sigma_{\ell_n}}.$$

\item {For any sequence $(\ell_n)_{n\geq 1}$ satisfying $\ell_n/n^{\rho/2}\to 0$, with high probability we have the following coupling
$$\SWG^v_{\sigma_i}= \widetilde{\sf BP}_{\theta_{i}} \quad \text{ for all }i\leq \ell_n,$$}
where the objects on the left and right are viewed as random trees with random edge lengths. 
\item Now for two sources, the shortest weight clusters from two randomly sampled vertices $v,u$ we can couple with high probability the corresponding shortest weight graphs $\SWG^v, \SWG^u$ with associated {\bf independent} branching process $\widetilde{\BP}^v, \widetilde{\BP}^u$ such that the above assertions (a),(b), (c) hold and further $\SWG^v_{\sigma_{\ell_n}} \cap \SWG^u_{\sigma'_{\ell_n}}=\emptyset.$
\end{enumeratea}

\end{prop}
\begin{proof}
Part(a) follows directly follows from Proposition \ref{coupling_forward_degrees}.  To prove (b) and (c) it is enough to show $\sigma_{\ell_n}\ll n^\rho$ with high probability. 

For $\ell \geq 1$, let 
$\Theta_\ell = |\mathcal{C}^{\BP}(u_{\sigma_\ell+1})|$ so that $\set{\Theta_i:i\geq 1}$ is an \emph{i.i.d.}~sequence with distribution $\chi$ as in Lemma \ref{size_cluster}. It follows from \eqref{BP_subseq} that, for $\ell \geq 1$, 
$${\Sigma}_{\ell} =|\mathcal{C}^{\BP}(u_{1})|+\sum_{i=1}^{\ell-1} \Theta_i.$$
 It is not difficult to see, from Lemma \ref{size_cluster}, that the distribution of $\set{\Theta_i:i\geq 1}$ is stochastically dominated by a collection of \emph{i.i.d.}~random variables $(\overline{X}_i)_{i\geq 1}$ with law
$$\pr( \overline{X}_i> x)=x^{-1/2}L(x),$$
for some slowly varying $L(\cdot)$. By standard weak convergence results to stable laws, see e.g. \cite[Theorem 3.8.2]{PTE5}, $\frac{\sum_{i=1}^{\ell-1}\overline X_i}{\ell^2}$ as $\ell \to \infty$ converges to a non-degenerate finite random variable with a stable law. It then follows that ${\Sigma}_{\ell_n} \ll n^{\rho}$ with high probability for $\ell_n\ll n^{\rho/2}$.
This, combined with the coupling in Proposition \ref{coupling_forward_degrees}, shows that with high probability the coupling between  $\SWG^v_i=\widetilde{\BP}[i]$ holds until $i\leq \Sigma_{\ell_n}$, on which the following coupling holds
$$(\sigma_i)_{i\leq \ell_n}=(\Sigma_i)_{i\leq \ell_n}.$$
This completes the proof.

\end{proof}
In subsequent proofs (see Section \ref{Step2}), we will pair the remaining free stubs, and it is crucial to have no prior knowledge of their weights. However, due to our construction, it is evident that $\SWG^v_{\sigma_{\ell_n}}$ only has stubs with positive weights emanating from it. To eliminate any issues arising from this prior knowledge, we take an additional step to pair all the stubs attached to $\SWG^v_{\sigma_{\ell_n}}$, which are conditioned to have strictly positive edge weights distributed as $F^* = F_{\zeta}$ as specified in \eqref{birth_time}. 

 Let the resulting cluster be denoted by $\overline{\SWG}^v_{\sigma_{\ell_n}} \subseteq [n]$, and define 
\beq\label{boundary}
\partial \SWG^v_{\sigma_{\ell_n}}=\overline{\SWG}^v_{\sigma_{\ell_n}}\backslash \SWG^v_{\sigma_{\ell_n}},
\eeq
to be the set of vertices lying at the boundary of $\SWG^v_{\sigma_{\ell_n}}$.

\begin{lem}\label{partial_SWG}
Let $\rho<\frac{\tau-2}{2\tau-3}$. For any $\eta \in (0,\rho)$ and a sequence $(\ell_n)_{n\geq 1}$ with $ \ell_n= \lfloor n^{\rho/2-\eta/4}\rfloor$, with high probability $n^{\rho-\eta}\ll |\partial \SWG^v_{\sigma_{\ell_n}}|\ll n^{\rho} $.
\end{lem}
\begin{proof}
As before let  $(\widetilde{\BP}[m])_{m\geq 1}$ be the discrete skeleton of the associated age dependent branching process and the sizes of the zero-weight clusters respectively as in \eqref{BP_subseq}.  
By Proposition \ref{coupling_SWG_BP} it suffices to show $n^{\rho-\eta}\ll|\partial\widetilde{\BP}[{\Sigma_{\ell_n}}]|\ll n^{\rho}$ w.h.p.

Recall the definition of $S_\chi$ from Lemma \ref{cluster_degree} and let $(S^{(j)})_{j\geq 1}$ be \emph{i.i.d.}~copies of $S_\chi$ (note that $S^{(j)}$ can depend on $(B_i)_{i\geq 2}$ but each $S^{(j)}$ involves disjoint subsets of $(B_i)_{i\geq 2}$).
As $|\partial\widetilde{\BP}[{\Sigma_i}]|$ represents the number of children with positive life-lengths attached to $\BP[{\Sigma_i}]$, it satisfies
$$|\partial \widetilde{\BP}{[\Sigma_{i+1}]}|=|\partial \widetilde{\BP}{[\Sigma_{i}]}|-1+S^{(i)} \quad \text{ for }i\geq 1.$$
For any $\ell\geq 1$, we have the following characterization
$$|\partial \widetilde{\BP}{[\Sigma_{\ell}]}|=|\partial \widetilde{\BP}[{\Sigma_1}]|+\sum_{i=1}^{\ell-1} (S^{(i)}-1).$$
It follows from Lemma \ref{cluster_degree} that $(S^{(i)}-1)_{i\geq 1}$ satisfies the stochastic ordering
$$\underline{X}_i \leq_{st} S^{(i)}-1 \leq_{st} \overline{X}_i$$
for two collections of \emph{i.i.d.}~random variables $(\underline{X}_i)_{i\geq 1}$ and $(\overline{X}_i)_{i\geq 1}$ such that 
$$\pr( \underline{X}_i> x)=x^{-1/2} L(x) \quad \text{and}\quad \pr( \overline{X}_i> x)=x^{-1/2}\tilde L(x)$$
for some slowly varying $L(\cdot),\tilde L(\cdot)$. Theorem 3.8.2 in \cite{PTE5} implies that $\frac{\sum_{i=1}^{\ell-1}\overline X_i}{\ell^2}$ converges to a non-degenerate random variable and hence yields that $|\partial\widetilde{\BP}[{\Sigma_{\ell_n}}]|\ll n^{\rho}$ w.h.p. thus proving the upper bound. 

 For the lower bound, it again follows from Theorem 3.8.2 in \cite{PTE5} that $\frac{\sum_{i=1}^{\ell-1}\underline X_i}{\ell^2}$ converges to a non-degenerate random variable $Y$ with stable law. Further $Y$ has an \textit{extremal} stable distribution with characteristic parameter $\alpha=1/2$ and by \cite[Proposition 7.16, P159]{kallenberg1997foundations} has a diffuse distribution and thus no mass at zero. 
 The lower bound $|\partial\widetilde{\BP}[{\Sigma_{\ell_n}}]|\gg n^{\rho-\eta}$ follows.

\end{proof}

\section{Proof of Theorem \ref{thm:main}} \label{sec:PfTh1}

In this Section we prove Theorem \ref{thm:main}. The lower bound on passage time is established in Section \ref{sec:lwb}, followed by the proof of the upper bound in  Section \ref{sec:upb}. The remaining results are proven in the next Section.

\subsection{Lower bound on passage time}
\label{sec:lwb}

Recall the sequence of stopping times for the associated continuous time branching processes $\set{\theta_{\ell}: \ell\geq 0}$ defined in \eqref{stoppingtime_BP}. Recall that $T_n(u,v)$ denotes the passage time between the two randomly sampled vertices $u,v$. 

\begin{prop}[Lower bound on passage time]
\label{prop:lwb}
Fix $\ell_n =\lfloor n^{\rho/2-\eta/4}\rfloor$ for $0 < \eta < \rho/2$. Under the assumptions of Theorem \ref{thm:main}, there exists a probability space where the distribution of $T_n(u,v)$ can be coupled with two random variables $\theta_{\ell_n}^{(v)}, \theta_{\ell_n}^{(u)}$ such that:
\begin{enumeratei}
\item $\theta_{\ell_n}^{(v)}, \theta_{\ell_n}^{(u)}$ are independent random variables with distribution $\theta_{\ell_n}$ as in \eqref{stoppingtime_BP}. 
\item With high probability as $n\to\infty$, $T_n(u,v) \geq \theta_{\ell_n}^{(v)} + \theta_{\ell_n}^{(u)}$. 
\end{enumeratei}
\end{prop}
The following is an immediate Corollary of the above Proposition. 
\begin{cor} 
\begin{enumeratea}
\item If the associated branching process {\bf does not} explode i.e., $\theta_\ell \uparrow \infty$ as $\ell \uparrow \infty$ then $T_n(u,v) \convp \infty$ as $n\to \infty$. 
\item If the associated branching processes explode in finite time, then for any $a> 0$, 
\[\liminf_{n\to \infty}\pr(T_n(u,v) \geq a) \geq \pr(\theta_\infty^{(u)}+ \theta_\infty^{(v)} \geq a),\]
where $\theta_\infty^{(u)}, \theta_\infty^{(v)}$ are independent copies of the explosion time. 
\end{enumeratea}

\end{cor}

\begin{proof}[Proof of Proposition \ref{prop:lwb}:]
It follows from Proposition \ref{coupling_SWG_BP}(d) that there exists two independent modified age-dependent branching processes $\widetilde{\sf BP}^v, \widetilde{\sf BP}^u$ such that they can be coupled respectively with $\SWG^v,\SWG^u$ with high probability. To be more precise, let $(\theta^{(v)}_\ell)_{\ell\geq 0}$, respectively $(\theta^{(u)}_\ell)_{\ell\geq 0}$, denote the stopping times defined as in \eqref{stoppingtime_BP} for $\widetilde{\sf BP}^v$, respectively $\widetilde{\sf BP}^u$. Then with high probability we have 
$$ \widetilde{\sf BP}^{v}_{\theta^{(v)}_{i}}=\SWG^v_{\sigma_i}, \quad \widetilde{\sf BP}^{v}_{\theta^{(u)}_{i}}=\SWG^u_{\sigma'_i} \quad \text{ for all }i\leq \ell_n,$$
and further the two shortest weight graphs are disjoint up to this stage with high probability. Thus 
$$T_n(v,u)\geq \theta^{(v)}_{\ell_n} +\theta^{(u)}_{\ell_n} $$
for $\ell_n =\lfloor n^{\rho/2-\eta/4}\rfloor$. 

\end{proof}

\subsection{Upper bound on passage time}
\label{sec:upb}
The proof of the upper bound on the typical passage time $T_n(u,v)$ involves identifying a path that connects 
$u$ and $v$ via the largest zero-weight cluster in the graph. 
Throughout this Section we will choose $\rho=3/8$ and let $\eta>0$ be a sufficiently small number. We fix the choice of $\ell$ as
\beq\label{choice_ell}
 \ell=\ell_n=\lfloor n^{\rho/2-\eta/4} \rfloor.
\eeq
There are three main steps in the proof.

\subsubsection{{\bf Step 1: Growing the SWGs to a proper size}}
Let $(\sigma_j)_{j\geq 1}$ and $(\sigma'_j)_{j\geq 1}$, be stopping times  for $\SWG^v$ and $\SWG^u$ defined as in \eqref{SWG_subseq}. We will grow both SWGs till $\SWG^v_{\sigma_\ell}$ and $\SWG^u_{\sigma'_\ell}$, and construct 
 $\overline{\SWG}^v_{\sigma_\ell},\overline{\SWG}^u_{\sigma'_\ell}$.
 Proposition \ref{coupling_SWG_BP} shows that with high probability we can construct two independent branching process $\widetilde {\sf BP}^v$ and $\widetilde {\sf BP}^u$ following Construction \ref{comparison_BP} that are coupled to $(\SWG^v_{i})_{1\leq i\leq \sigma_{\ell_n}}$ and $(\SWG^u_{i})_{1\leq i\leq \sigma'_{\ell_n}}$ respectively and such that with high probability 
$\overline{\SWG}^v_{\sigma_\ell}\cap \overline{\SWG}^u_{\sigma'_\ell}=\emptyset.$

Fix $\eps> 0$ and define, 

$$S^\ep(v):=\{ x\in \partial \SWG^v_{\sigma_{ \ell}}: \exists w \in \SWG^v_{\sigma_{ \ell}} \text{ with }  t( x, w)\leq \ep\},$$
namely the set of vertices on the boundary that are connected to $\SWG^v_{\sigma_{ \ell}}$ with a direct edge of length $\ep$. Note that by construction, 
\[S^\ep(v) \subseteq \{ x\in \partial \SWG^v_{\sigma_{ \ell}}: T_n( x, \SWG^v_{\sigma_{ \ell}})\leq \ep\}. \]
 Next, to ease notation, write,
$$S^{\ep,c}(v):=\partial \SWG^v_{\sigma_{ \ell}}\backslash S^\ep(v),$$
for the ``complement'' of vertices in $\partial \SWG^v_{\sigma_{ \ell}}$ namely vertices in the boundary which do not have a \emph{direct} connection of length $\leq \eps$ to $\SWG^v_{\sigma_{ \ell}}$.
Define the ``degree'' of the set $S^\ep(v)$, denoted by $\deg(S^\ep(v))$, for the total number of stubs attached to vertices in $S^\ep(v)$. For any subset of vertices $A\subseteq S^\ep(v)$, similarly let $\deg(A)$ denote the total number of stubs attached to vertices in $A$.

\begin{lem}\label{boundary_size}
For any $\ep>0$ and $\eta\in (0,\rho)$, whp  
 $n^{\rho-\eta}\ll \deg(S^\ep(v)) \ll n^{\rho}$. 
\end{lem}
\begin{proof}
Proposition \ref{coupling_forward_degrees} implies that w.h.p. we can construct $\SWG^v_{\sigma_{ \ell}}$ by sampling \emph{i.i.d.}~random variables $(B_i)_{i\geq 2}$ from the size-biased distribution $\vq$. Furthermore, it follows from a straightforward adaption of the proof of Proposition \ref{coupling_forward_degrees} that whp the forward degrees of vertices in $\partial \SWG^v_{\sigma_{ \ell}}$ can also be coupled with \emph{i.i.d.}~$(\tilde B_i)_{i\geq 1}$ from $\vq$. 
Note that since each vertex in the boundary has at lease one edge to $\SWG^v_{\sigma_{ \ell}}$, thus $S^\ep(v) \sustod \mathrm{Binomial}(\partial \SWG^v_{\sigma_{ \ell}},F^*(\ep))$ where $\partial \SWG^v_{\sigma_{ \ell}}\gg n^{\rho-\eta}$ according to Lemma \ref{partial_SWG} and as before $\sustod$ denotes the stochastic domination ordering. When the aforementioned coupling to \emph{i.i.d.}~forward degrees holds, we have 
$$ \deg(S^\ep(v)) =\sum_{x\in S^\ep(v)} \widetilde B_x.$$
It is then easy to see  $n^{\rho-\eta}\ll \deg(S^\ep(v)) \ll n^{\rho}$ with high probability. 
\end{proof}

\subsubsection{{\bf Step 2: The largest zero-weight cluster in the rest of the graph}}
\label{Step2}
\begin{construction}[Modified degree sequence $\widetilde \vd$]\label{modified_degree_construction}
The modified degree of each vertex in $[n]$ is defined to be its number of remaining free stubs after the construction of $\overline{\SWG}^v_{\sigma_\ell}\cap \overline{\SWG}^u_{\sigma'_\ell}$ in Step 1.  
\end{construction}

Note that the modified degrees of vertices in $ \SWG^v_{\sigma_{ \ell}}\cup  \SWG^u_{\sigma'_{ \ell}}$ are all zeros, and the modified degrees of vertices in $\partial \SWG^v_{\sigma_{ \ell}}\cup \partial \SWG^u_{\sigma'_{ \ell}}$ are \emph{i.i.d.}~samples from $\vq$ which have been revealed in Step 1. Throughout this section we will assume $\widetilde\vd$ satisfies 
\beq\label{event_dmax}
 \max_{x\in \partial \SWG^v_{\sigma_{ \ell}}\cup \partial \SWG^u_{\sigma'_{ \ell}}}\widetilde d (x) \leq n^{\rho/(\tau-2)}.
 \eeq
It is easy to see this condition holds with high probability based on Lemma \ref{partial_SWG}.
Next consider the unexplored vertices 
$$\cR_n:= [n]\setminus (\overline{\SWG}^v_{\sigma_\ell}\cap \overline{\SWG}^u_{\sigma'_\ell}).$$
Throughout this Section we will treat $|\cR_n|$ as known and condition on the event $|\cR_n|=n-o(n^{\rho})$, which occurs with high probability.
As a result of Proposition \ref{coupling_forward_degrees}, with high probability the modified degrees of vertices in $\cR_n$ are given by  \emph{i.i.d.}~samples from $\vp$, which are unknown to us yet.  

To obtain the zero-weight clusters in the rest of the graph, we essentially run a bond percolation on $\CM_n(\widetilde\vd)$ where each edge is independently retained with probability $p_c$ (otherwise the edge is deleted). We will denote the resulting bond percolation clusters by $\CM_n(\widetilde\vd,p_c)$. For each vertex, its \textit{thinned degree} is defined to be its induced degree in $\CM_n(\widetilde\vd,p_c)$, which represents the number of zero-weight edges attach to it. Let $\cC_{(1)}$ denote the largest cluster in $\CM_n(\widetilde\vd,p_c)$ and write $\cP^+(\cC_{(1)})$ for the number of positive-weighted stubs emanating from vertices in $\cC_{(1)}$. Our goal in this Step is to prove  
\beq\label{positive_stubs}
\cP^+(\cC_{(1)})=\Omega_{\pr}(n^{2/3}).
\eeq

This conclusion essentially follows from the results in \cite{dhara2017critical} for the critical percolation $\CM_n(\widetilde\vd,p_c)$. We begin by summarizing in Assumption \ref{modified_ass} the assumptions made in \cite{dhara2017critical} (i.e., \cite[Assumption 3.5]{dhara2017critical}). Note that the assumptions are stated for a fixed degree sequence whereas we are working with a random (modified) degree sequence $\widetilde\vd$. Hence we will show that $\widetilde\vd$ satisfies Assumption  \ref{modified_ass} with high probability.

\begin{ass}\label{modified_ass}
For a given degree sequence $\vd=(d_i)_{i\in [n]}$, let $D_n$ denote the degree of a uniformly chosen vertex in $[n]$. Suppose $D_n$ satisfies
\begin{enumeratei}
\item $D_n\convd D$;
\item $\E[ D_n^3]\to \E[D^3]$;
\item $$\nu_n:=\frac{\E[ D_n( D_n-1)]}{\E  D_n}\to \nu=\frac{\E[D(D-1)]}{\E D}>1;$$
\item (Critical window for percolation) $p_c=\frac{1}{\nu_n}+O(n^{-1/3})$.
\end{enumeratei}
\end{ass}

\begin{remark}\label{ass_clarification}
The assumptions above implicitly involve an infinite sequence $\vd=(d_i)_{i\geq 1}$, whereas we are working with a sequence of (random) modified degrees $\{\widetilde \vd^{(n)}:n\geq 1\}$ where $\widetilde \vd^{(n)}=(\widetilde d^{(n)}_i)_{i\in [n]}$ is defined for $\CM_n(\vD)$. To avoid any ambiguity, we will perform the following construction. First sample $\vd:=(d_i)_{i\geq 1}$ where $d_i\sim_{i.i.d.}\vp$. Based on the discussion above, for each $n\geq 1$ we can couple $(\widetilde d^{(n)}_u)_{u\in \cR_n}$ to the first $|\cR_n|$ terms in $\vd$, i.e.,  $(\widetilde d^{(n)}_u)_{u\in \cR_n}=(d_i)_{i\leq |\cR_n|}$. 
\end{remark}

\begin{lemma}
Let $\{\widetilde \vd^{(n)}: n\geq 1\}$ denote the collection of modified degree sequences in $\{ \CM_n(\vD): n\geq 1\}$. With high probability $\{\widetilde \vd^{(n)}: n\geq 1\}$ satisfies Assumption \ref{modified_ass}.
\end{lemma}
\begin{proof}
Without loss of generality we can condition on \eqref{event_dmax} and treat $|\cR_n|$ as a known constant satisfying $|\cR_n|=n-o(n^\rho)$. Following Remark \ref{ass_clarification}, we define $\vd=(d_i)_{i\geq 1}$ so that the coupling $(d_i)_{i\leq |\cR_n|}=(\widetilde d^{(n)}_u)_{u\in \cR_n}$ holds for all $n\geq 1$. For simplicity, we will omit the superscript $(n)$ moving forward, as the dependence on $\CM_n(\vD)$ is clear.

Let $\widetilde D_n$ denote the modified degree $\widetilde d_U$ of a uniformly chosen vertex $U\in [n]$, whereas let $D_n$ denote the modified degree $\widetilde d_{U'}$ of a uniformly chosen vertex $U'\in \cR_n$. Due to the coupling $(\widetilde d^{(n)}_u)_{u\in \cR_n}=(d_i)_{i\leq |\cR_n|}$, $D_n$ has the law of a uniform sample from $(d_i)_{i\leq |\cR_n|}$.

We begin by considering the convergence of $D_n$. Strong law of large numbers immediately implies that Assumption \ref{modified_ass} (i)-(iii) are satisfied for $D_n$ almost surely. In particular, we have $\E[D_n^j]\to \E[D^j]$ almost surely for $1\leq j\leq 3$, noting that $\E[D_n^j]:=\frac{ \sum_{i\leq |\cR_n|} d^j_i}{|\cR_n|}$ is a random variable arising from the randomness of $(d_i)_{i \leq |\cR_n|} \sim_{i.i.d.} \vp$. Furthermore, we'd like to show that there exists $n_0\in\mathbb{N}$ and $C_{1},C_{2}>0$, such that with high probability we have
\beq\label{empirical_concentration}
\big|\E D_n-\E D\big|\leq C_{1}n^{-1/3}\quad \text{ and }\quad \big|\E[D_n^2]-\E[D^2]\big|\leq C_{2}n^{-1/3}
\eeq
for all $n\geq n_0$. These two conclusions follow from applying Corollary 3 of \cite{tang2022convergence} with $X_k=\widetilde d_k$, $\alpha=1$, and respectively with $X_k=(\widetilde d_k)^2$, $\alpha=(1+\eta)/2$ for $\eta$ as in Assumption \ref{ass} (A2).

We are now ready to verify Assumption \ref{modified_ass} for $\widetilde D_n$.
It is straightforward to see that $\widetilde D_n$ satisfies Assumption \ref{modified_ass} (i) whp since $D_n$ satisfies it almost surely. To prove (ii), observe that conditioning on \eqref{event_dmax},
\beq\label{empirical_comparison}
\frac{|\cR_n|}{n}\cdot \E[D_n^3]\leq \E[\widetilde D_n^3]\leq \frac{|\cR_n|}{n} \cdot \E[D_n^3]+\frac{O(n^{\rho})}{n} (n^{\rho/(\tau-2)})^3,
\eeq
and it follows that  $\E[\widetilde D_n^3]\to \E[D^3]$ whp by the choice of $\tau>4, \rho=3/8$. Based on the same reasoning, one has $\E[\widetilde D_n^2]\to \E[D^2], \E[\widetilde D_n]\to \E D$ as $\E[D_n^2]\to \E[D^2], \E[D_n]\to \E[D]$. Hence,
$$\widetilde\nu_n:=\frac{\E[\widetilde D_n(\widetilde D_n-1)]}{\E \widetilde D_n}\to \nu=\frac{\E[D(D-1)]}{\E D}>1,$$
verifying (iii) for $\widetilde D_n$.

Lastly we need to show with high probability  $\CM_n(\widetilde\vd,p_c)$ belongs to the critical window, i.e., 
$$|p_c-\frac{1}{\widetilde\nu_n}|=O(n^{-1/3}).$$
It suffices to show $|\widetilde\nu_n-\nu|=O(n^{-1/3})$. The triangle inequality implies
\begin{align*}
|\widetilde\nu_n-\nu|&\leq \bigg|\widetilde \nu_n-\frac{\E[ \widetilde D_n(\widetilde D_n-1)]}{\E D}\bigg|+\bigg|\frac{\E[ \widetilde D_n(\widetilde D_n-1)]}{\E D}-\frac{ \E(D(D-1))}{\E D}\bigg|\\
&=O( |\E \widetilde D_n-\E D|)+O( |\E [\widetilde D^2_n]-\E [D^2]|).
\end{align*}
It follows from the same reasoning as in \eqref{empirical_comparison} that $\E \widetilde D_n^2=\E D_n^2+O(n^{\rho-1})$ and $\E \widetilde D_n=\E D_n+O(n^{\rho-1})$ whp. Combined with \eqref{empirical_concentration}, we have that whp
\begin{align*}
|\E \widetilde D_n-\E D|&\leq C_{1} n^{-1/3}+O(n^{1-\rho})=O(n^{-1/3})\\
|\E [\widetilde D^2_n]-\E [D^2]|&\leq C_{2}n^{-1/3}+O(n^{1-\rho})=O(n^{-1/3}),
\end{align*}
thus proving that $|\widetilde\nu_n-\nu|=O(n^{-1/3})$.
\end{proof}
Having verified that whp $\CM_n(\widetilde\vd,p_c)$ satisfies the assumptions in \cite{dhara2017critical} for critical percolation, we now apply their results to complete the proof of \eqref{positive_stubs}.

\noindent
\textit{Proof of \eqref{positive_stubs}.}
Let $N_k(\cC_{(1)})$ be the number of vertices in $\cC_{(1)}$ with thinned degree $k$.  It is straightforward to observe that 
$$\cP^+(\cC_{(1)})\geq N_1(\cC_{(1)})$$
as each vertex with thinned degree 1 has at least one positive-weight stub attach to it.  It follows from Theorem 3.6 of \cite{dhara2017critical} that $|\cC_{(1)}|=\Omega_{\pr}(n^{2/3})$ and Eqn. (6.4) that 
$$N_1(\cC_{(1)})=\Omega_{\pr}(|\cC_{(1)}|)=\Omega_{\pr}(n^{2/3}).$$

\qed

\subsubsection{{\bf Step 3: Small connection time between  the shortest weight clusters and the largest zero-weight cluster }}
Without loss of generality we only need to show that with high probability the passage time between $\SWG^v_{\sigma_{ \ell}}$ and $\cC_{(1)}$ is at most $\ep$ for arbitrarily small $\ep>0$. The conclusion for $\SWG^u_{\sigma'_{ \ell}}$ can be derived analogously and hence w.h.p.
$$T_n(u,v)\leq \theta^{(v)}_{\infty}+\theta^{(u)}_{\infty}+2\ep.$$

\begin{prop}
Let $\rho=3/8$ and $\ell=\lfloor n^{\rho/2-\eta/4} \rfloor$. For any $\ep>0$ and $w\in \{u,v\}$ there exists a path of weight less than $\ep$ that connects $\SWG^w_{\sigma_{ \ell}}$ to $\CC_{(1)}$ with high probability.
\end{prop}

\begin{proof}
First note that if one of the zero stubs attached to $S^\ep(v)$ has been used to form $\cC_{(1)}$, then this immediately  implies that $\SWG^v_{\sigma_{ \ell}}$ is connected to $\cC_{(1)}$ within time $\ep$.

Otherwise we will pair the positive stubs uniformly at random. For each positive stub we flip an independent coin with success probability $\sqrt{F^*(\ep)}$ and say a positive stub is $\ep$-good if we have a success. Our goal is to show w.h.p. there exists an $\ep$-good edge between $S^\ep(v)$ and $\cC_{(1)}$ so that $\SWG^v_{\sigma_{ \ell}}$ is connected to $\cC_{(1)}$ within time $2\ep$. It follows from \eqref{positive_stubs} that there are $\Omega_{\pr}(n^{2/3})$ positive stubs attached to $\cC_{(1)}$. 
By standard large deviations for Binomial distribution we see that with high probability $S^\ep(v)$ has $\Omega_{\pr}(n^{\rho-\eta})$ $\ep$-good stubs and $\cC_{(1)}$ has $\Omega_{\pr}(n^{2/3})$  $\ep$-good stubs.

Let $L_n$ denote the total number of positive stubs, and note that $L_n\asymp n$. The probability that none of the $\ep$-good stubs of $S^\ep(v)$'s is paired to that of $\cC_{(1)}$'s is at most 
$$\left(1-\frac{\Omega(n^{2/3})}{L_n}\right)^{\Omega(n^{\rho-\eta})}\asymp e^{-\Omega(n^{\rho-\eta-1/3})}=o(1)$$
since $\rho=3/8$.
Thus whp, for any $\eps>0$, there exists a path from $u,v$ with passage time $\leq \theta_{\ell_n}^{(u)} + \theta_{\ell_n}^{(v)} + 2\eps +o_{\pr}(1)$. This completes the proof.

\end{proof}

\section{Proofs of the remaining results}
\label{sec:pf-remaining}
Here we prove the rest of the results in the paper as well provide a proof idea for Conjecture \ref{lb_passage_time}. 

\subsection{Proof of Corollaries \ref{cor:examples}-\ref{dexp}}
\label{sec:pf-corr}

\noindent
\textit{Proof of Corollary \ref{cor:examples}.} The results follow from Theorem \ref{thm:main} by verifying that the edge distributions are min-summable in (a)-(c). For $a,b>0$, let $F_a(\cdot), G_b(\cdot)$ be defined as in \eqref{eqn:Fa-def}, \eqref{eqn:Gb-def}. Let  the positive part of $F_a(\cdot)$ be denoted by
$$F_{a,+}(x)=\frac{F_a(x)-F_a(0)}{1-F_a(0)}=\frac{x^a}{1-p_c} \quad \text{ for }0\leq x^a\leq 1-p_c.$$
Thus  $F_{a,+}^{-1}(e^{-u})=(1-p_c)^{1/a}e^{-u/a}$, which leads to
$$\int_{1/\ep}^\infty F_{a,+}^{-1}(e^{-u}) u^{-1} \; du<\infty$$
for any $a>0$.

Similarly, one can obtain $G_{b,+}(x)=(1-p_c)^{-1}\exp(-1/x^b)$ for $0\leq \exp(-1/x^b)\leq 1-p_c$ and $G_{b,+}^{-1}(e^{-u})=(u-\log(1-p_c))^{-1/b}$, which is min-summable for all $b>0$.
\qed

\bigskip
\noindent
\textit{Proof of Corollary \ref{dexp}.}
The positive part of the edge-weight distribution is denoted by  $H_{\gamma}(t)=e\cdot \exp( -\exp(1/t^{\gamma}))$. Then $H_\gamma^{-1}(e^{-u})=(\log(1+u))^{-1/\gamma}$, which leads to 
$$\int_{1/\ep}^\infty (\log(1+u))^{-1/\gamma}u^{-1} \; du
\begin{cases}
<\infty & \text{ when } \gamma<1, \\
=\infty & \text{ when }\gamma\geq 1.
\end{cases}$$
The result follows from Theorem \ref{thm:main}.
\qed

\subsection{Proof of Theorem \ref{thm:extremal}:}

Let $C\in (0, \frac{1}{-2\log q_1})$. It follows from Proposition \ref{isolated_path} that with high probability there exists an isolated path $\Gamma$ of length $\lfloor C\log n\rfloor$. Let $d(\cdot,\cdot)$ denote the graph distance. Let $v,v'$ denote the two end points of this isolated path and $v_{\mathrm{mid}}$ a mid point such that $\max\{ d(v,v_{\mathrm{mid}}),d(v',v_{\mathrm{mid}})\}\geq  \lfloor \frac{\lfloor C\log n\rfloor}{2}\rfloor$.

 In order to reach all vertices in $\mathcal{G}_n$, the first passage percolation starting from vertex $u=1$ has to go through an isolated path of length $L=\lfloor \frac{\lfloor C\log n\rfloor}{2}\rfloor$.

In particular, let $\set{t_i:i\geq 1}$ be an i.i.d. collection of random variables with the edge weight distribution,  with high probability,   $\sum_{i=1}^L t_{i}\stod  \max_{v\in [n]}  T_n(1,v) $ where $\stod$ denotes the stochastic domination operation (i.e., we can construct the two random variables above on a common probability space such that whp $\sum_{i=1}^L t_{i} <  \max_{v\in [n]}  T_n(1,v) $).
The law of large numbers guarantees that whp for some $c > 0$,
$$\sum_{i=1}^L t_{i}\geq \frac{\E[t_e]L}{2}\geq c\log n.$$

\qed 

\subsection{Proof strategy for Conjecture  \ref{lb_passage_time}:} We will outline an argument that shows, in the setting of Corrolary \ref{dexp}(a), when $\gamma > 1$ then $T_n(u,v)$ grows at least like $(\log\log(n))^{1-1/\gamma}$. A full rigorous proof will require showing that the various coupling steps below can be carried out.   Let $u,v \in [n]$ be two vertices chosen independently and uniformly at random. To establish a lower bound for $T_n(u,v)$, the key idea is to construct first passage percolation clusters centered at $u$ and $v$, grown for enough time, but not too large so that the clusters are still disjoint since in this case $T_n(u,v)$ is bounded from below by the passage times needed to reach the boundaries of these clusters. If one can estimate (at least lower bound) the passage time to these boundaries then this would also result in a lower bound on $T_n(u,v)$.  

One tractable approach is as follows.  Grow the SWG from $v$ while collapsing every zero-weight cluster formed during the growth into a single point. We will only grow the (uncollapsed) $\SWG^v$ and $\SWG^u$ till size $O(n^\rho)$ for sufficiently small $\rho>0$ to ensure the two clusters are disjoint, and further these can be coupled to branching processes.
Following a similar argument to that of Proposition \ref{coupling_SWG_BP}, one should be able to obtain the coupling between the collapsed SWG of $v$ and $\BP(F^*,h^*)$ up to a certain generation $L=O(\log\log n)$ (here we view $\BP(F^*,h^*)$ as a discrete time process where generation represents the distance to the root). Instead, we can write $\mathcal{Z}:=\{\mathcal{Z}_m\}_{m\geq 1}$ as the discretized version of $\BP(F^*,h^*)$ where $\mathcal{Z}_m$ is the size of generation $m$. Letting $X$ denote the offspring distribution of $\mathcal Z$,  Lemma \ref{cluster_degree} shows that $\pr(X\geq x)\sim x^{-1/2}$. It then follows from \cite{Davies} that there exists a non-degenerate random variable $W$ such that 
\beq\label{BP_convergence}
2^m \log(1+\mathcal{Z}_m)\overset{a.s.}{\to}W
\eeq
where $W>0$ on survival. Survival occurs with probability one in our setting owing to Assumptions \ref{ass}. This allows us to estimate the size of $\BP(F^*,h^*)$ up to a certain generation, which then can be connected to the size of $\BP(F,h)$. Exploiting this connection yields that we should take  $L=O(\log\log n)$ to ensure the coupling between the collapsed SWG and $\mathcal{Z}$ holds up to generation $L$ and such that the clusters from the two vertices are still disjoint whp. 

It then remains to understand the passage time required for traveling from one of the sources, say $v$,  to vertices in the $L$-th generation of $\mathcal{Z}$. As a result of \eqref{BP_convergence}, for any $\ep>0$, there exists some $M:=M_\ep\in \mathbb{R}$ such that with probability at least $1-\ep$, $|(1/2)^m \log(1+\mathcal{Z}_m)-W|\leq 1$ for all $m\geq M$. Let $K:=K_\ep\in \mathbb{R}$ be such that $\pr(W>K)\leq \ep$. Then with probability at least $1-2\ep$, we have 
 $$\mathcal{Z}_m \leq \exp( (1+K)2^m) \quad \text{ for all }m\geq M_\ep.$$
Set $a_m:=\exp( (1+K)2^m)$ and let $\{\sigma_{m,i}: m\in\mathbb{N}, i\in [a_m]\}$ be a collection of i.i.d. edge weights following distribution $F^*$. It is straightforward to see that the shortest time to travel from generation $m$ to generation $m+1$ stochastically dominates
$$t_m:=\min_{i\in [a_m]} \sigma_{m,i}$$
and the passage time from $v$ to generation $L$ dominates $\sum_{m=M}^L t_m$.
Taking $s_m>0$ to be such that $F^*(s_m)\approx \exp(-(2+K)2^m)$, we have
\begin{align*}
\pr(t_m\leq s_m)&=1-(\pr(\sigma_{m,1}>s_m))^{a_m}=1-(1-F^*(s_m))^{a_m}\leq F^*(s_m)a_m\leq \exp(-2^m).
\end{align*}
One can choose $M$ to be sufficiently large so that with a large probability $t_m\leq s_m$ for all $m\geq M$ and hence with high probability, 
$$T_n(v,u) {\sustod} \sum_{m=M}^{L}t_m\geq \sum_{m=M}^{L}s_m=\sum_{m=M}^{L} [F^*]^{-1}(\exp (-(2+K)2^m)),$$
where once again $\sustod$ represents the stochastic domination ordering. 
In particular, when $F^*=H_\gamma$, $[F^*]^{-1}(\exp (-(2+K)2^m))\geq (\kappa m)^{-1/\gamma}$ for some $\kappa>0$ and 
$$\sum_{m=M}^{L} [F^*]^{-1}(\exp (-(2+K)2^m))\geq C L^{1-1/\gamma}.$$
Since we have taken $L\asymp \log\log n$,  and assuming the above coupling can be carried out, the above bound  would complete the proof.

\section*{Acknowledgements}
Bhamidi was partially supported by NSF DMS-2113662, DMS-2413928, and DMS-2434559 and NSF RTG grant DMS-2134107. Durrett was partially supported by NSF DMS 2153429.

\bibliography{refs.bib}

\end{document}